\documentclass{amsart}

\newcommand{\mint}{\mathop{\int\hspace{-1.05em}{\--}}\nolimits}

\newcommand{\R}{{\mathbb R}}
\newcommand{\N}{{\mathbb N}}

\newcommand{\eps}{\varepsilon}

\newtheorem{definition}{Definition}[section]
\newtheorem{rem}[definition]{Remark}

\newtheorem{thm}[definition]{Theorem}
\newtheorem{pro}[definition]{Proposition}
\newtheorem{lem}[definition]{Lemma}
\newtheorem{cor}[definition]{Corollary}
\makeatletter

\begin{document}
\parindent0mm

\title[Curvature functionals with superquadratic growth]{Two-dimensional curvature functionals with superquadratic growth}

\author{Ernst Kuwert}
\address[E.~Kuwert]{Mathematisches Institut\\Universit\"at Freiburg\\Eckerstra\ss{}e 1\\79104 Freiburg\\Germany}
\email{ernst.kuwert@math.uni-freiburg.de}
\author{Tobias Lamm}
\address[T.~Lamm]{Institut f\"ur Mathematik\\Goethe-Universit\"at Frankfurt\\Robert-Mayer-Str. 10\\60054 Frankfurt\\Germany}
\email{lamm@math.uni-frankfurt.de}
\author{Yuxiang Li}
\address[Y.~Li]{Department of Mathematical Sciences\\Tsinghua University\\Beijing 100084\\P.R.China}
\email{yxli@math.tsinghua.edu.cn}
\thanks{This research was supported in parts by the DFG Collaborative Research Center SFB/Transregio 71, 
and by a PIMS Postdoctoral Fellowship of the second author.}
\date{\today}
\begin{abstract}
For two-dimensional, immersed closed surfaces $f:\Sigma \to \R^n$, we 
study the curvature functionals $\mathcal{E}^p(f)$ and $\mathcal{W}^p(f)$
with integrands $(1+|A|^2)^{p/2}$ and $(1+|H|^2)^{p/2}$, respectively. Here 
$A$ is the second fundamental form, $H$ is the mean curvature and we assume 
$p > 2$. Our main result asserts that $W^{2,p}$ critical points are smooth
in both cases. We also prove a compactness theorem for $\mathcal{W}^p$-bounded 
sequences. In the case of $\mathcal{E}^p$ this is just Langer's theorem
\cite{langer85}, while for $\mathcal{W}^p$ we have to impose a bound for 
the Willmore energy strictly below $8\pi$ as an additional condition. 
Finally, we establish versions of the Palais-Smale condition for
both functionals.
\end{abstract}
\maketitle

\section{Introduction}
Let $\Sigma$ be a two-dimensional, closed differentiable manifold and $p > 2$,
hence $W^{2,p}(\Sigma,\R^n) \subset C^{1,1-\frac{2}{p}}(\Sigma,\R^n)$ by the
Sobolev embedding theorem. On the open subset of immersions $W^{2,p}_{{\rm im}}(\Sigma,\R^n)$ 
we consider the two functionals 
\begin{eqnarray*}
\label{defEp}
\mathcal{E}^p(f) & = & \frac{1}{4} \int_\Sigma (1+|A|^2)^{\frac{p}{2}}\,d\mu_g,\\ 
\label{defWp}
\mathcal{W}^p(f) & = & \frac{1}{4} \int_\Sigma (1+|H|^2)^{\frac{p}{2}}\,d\mu_g.
\end{eqnarray*}
Here $g$ denotes the first fundamental form with induced 
measure $\mu_g$, $A = (D^2 f)^\perp$ the second
fundamental form, and $H$ is the mean curvature vector.
We prove regularity of critical points for both functionals.

\begin{thm}\label{regularity}
Let $f \in W^{2,p}_{{\rm im}}(\Sigma,\R^n)$ be a critical point of $\mathcal{W}^p$
or ${\mathcal E}^p$, where $2 < p < \infty$. Then local graph representations of $f$ are smooth, in fact real analytic.
\end{thm}

In a graph representation, the Euler-Lagrange equations 
become fourth order elliptic systems, where the 
principal term has a double divergence structure. The systems 
are degenerate, in the sense that in both cases the coefficient of the 
principal term involves a $(p-2)$-th power of the curvature, which 
a priori may not be bounded.
For the functional $\mathcal{W}^p(f)$, our first step towards 
regularity is an improvement of the integrability of $H$. For this 
we employ an iteration based on a new test function argument. More
precisely, we solve the equation $L_g \varphi = |H|^{\lambda-1} H$ for 
appropriate $\lambda > 1$ and then insert $\varphi$ as a test function.
Here the operator 
$L_g = \sqrt{\det g} g^{\alpha \beta} \partial^2_{\alpha \beta}$ 
comes up in the principal term of the equation.\\
\\
Unfortunately, 
the same strategy does not apply in the case of the functional 
$\mathcal{E}^p(f)$, since then the corresponding operator is 
a full Hessian and hence the equation would be overdetermined. Instead 
we first use a hole-filling argument to show power decay for the
$L^p$ integral of the second derivatives, and derive $L^2$
bounds for the third derivatives by a difference-quotient 
argument; these steps follow closely the ideas of Morrey
\cite{morrey66} and L. Simon \cite{simon93}. In the final
critical step we adapt a Gehring type lemma due to Bildhauer, Fuchs
and Zhong \cite{bildhauer05} as well as the Moser-Trudinger 
inequality to get that the solution is of class $C^2$. Since 
it is also not immediate how to modify the $\mathcal{E}^p(f)$ 
approach to cover the functional $\mathcal{W}^p(f)$,  
we decided to include both independent arguments. 

As second issue we address the existence of minimizers for the functionals. 
By the compactness theorem of Langer \cite{langer85}, sequences of closed
immersed surfaces $f_k:\Sigma \to \R^n$ with $\mathcal{E}^p(f_k) \leq C$
subconverge weakly to an $f \in W^{2,p}_{{\rm im}}(\Sigma,\R^n)$, after
suitable reparametrization and translation. In particular, we obtain
the existence of a smooth ${\mathcal E}^p$ minimizer in the class
of immersions $f:\Sigma \to \R^n$ for $p > 2$. On the other hand, 
boundedness of $\mathcal{W}^p(f)$ is not sufficient to guarantee the
required compactness. This is easily illustrated by joining two
round spheres by a shrinking catenoid neck, showing that the 
$8\pi$ bound in the following result is optimal. 

\begin{thm} \label{comp} Let $\Sigma$ be a closed surface and 
$f_k \in W^{2,p}_{{\rm im}}(\Sigma,\R^n)$ be a sequence of 
immersions with $0 \in f_k(\Sigma)$ and 
$$
\mathcal{W}^p(f_k) \leq C \quad \mbox{ and } \quad
\liminf_{k \to \infty} \frac{1}{4} \int_{\Sigma} |H_k|^2\,d\mu_{g_k} < 8\pi.
$$
After passing to $f_k \circ \varphi_k$ for appropriate 
$\varphi_k \in C^\infty(\Sigma,\Sigma)$ and selecting a subsequence,
the $f_k$ converge weakly in $W^{2,p}(\Sigma,\R^k)$ to an 
$f \in W^{2,p}_{{\rm im}}(\Sigma,\R^n)$. In particular, the 
convergence is in $C^{1,\beta}(\Sigma,\R^n)$ for any 
$\beta < 1-\frac{2}{p}$ and we have 
$$
\mathcal{W}^p(f) \leq \liminf_{k \to \infty} \mathcal{W}^p(f_k).
$$
\end{thm}

A classical approach to the construction of harmonic maps,
due to Sacks \& Uhlenbeck \cite{sacks81}, is by introducing perturbed 
functionals involving a power $p > 2$ of the gradient. One 
motivation for our analysis is an analogous approximation 
for the Willmore functional 
\begin{align}
\label{defwillmore}
\mathcal{W}(f) = \frac{1}{4} \int_\Sigma |H|^2\, d\mu_g 
= \frac{1}{4} \int_\Sigma |A|^2\, d\mu_g + \pi \chi(\Sigma).
\end{align}
The Willmore functional does not satisfy a Palais-Smale type
condition, since it is invariant under the group of M\"obius
transformations. In Section \ref{sectionps} we 
verify suitable versions of the Palais-Smale condition for 
the functionals $\mathcal{E}^p$ and $\mathcal{W}^p$ with
$p > 2$. In a forthcoming paper, we study the limit $p \searrow 2$ 
in the case of $\mathcal{E}^p(f)$, proving a concentration 
compactness alternative and a partial blowup analysis.\\ 
\\
Curvature functionals with nonquadratic growth appear
also in the work of Bellettini, Dal Maso and Paolini
\cite{BDP93} as well as Ambrosio and Masnou \cite{ambrosio03}.
However their focus is much different, for instance the latter
paper is motivated by applications to image restoration.

\section{The Euler Lagrange Equations}\label{euler}
Here we compute in local coordinates the Euler Lagrange equations of the 
functionals $\mathcal{E}^p(f)$ and $\mathcal{W}^p(f)$. For an immersed
surface the fundamental forms are 
$$
g_{\alpha \beta} = \langle \partial_\alpha f,\partial_\beta f \rangle 
\quad \mbox{ and } \quad 
A_{\alpha \beta} = P^\perp (\partial^2_{\alpha \beta} f). 
$$
Here $P^\perp$ is the projection onto the normal space given by 
$$
P^\perp =  {\rm Id} - g^{\alpha \beta} \langle \partial_\alpha f, \cdot \rangle  \partial_\beta f.
$$
We compute further 
$$
|A|^2 = 
g^{\alpha \gamma} g^{\beta \lambda} 
\big\langle P^\perp \partial^2_{\alpha \beta} f,\partial^2_{\gamma \lambda} f \big\rangle
\quad \mbox{ and } \quad 
|H|^2 = 
g^{\alpha \beta} g^{\gamma \lambda}              
\big\langle P^\perp \partial^2_{\alpha \beta} f,\partial^2_{\gamma \lambda} f \big\rangle.
$$
On the open set of $W^{2,p}$ immersions, both $\mathcal{E}^p$ and $\mathcal{W}^p$
are differentiable in the sense of Fr\'{e}chet. The derivative of $\mathcal{E}^p$ is given by
\begin{eqnarray*}
\label{Ederivative}
D{\mathcal E}^p(f)\phi & = & \frac{p}{4} \int_\Sigma (1+|A|^2)^{\frac{p-2}{2}} 
g^{\alpha \gamma} g^{\beta \lambda} 
\big\langle P^\perp \partial_{\alpha \beta}^2 f, \partial^2_{\gamma \lambda} \phi \rangle \sqrt{\det g}\\
\nonumber
& + & \frac{p}{8} \int_\Omega (1+|A|^2)^{\frac{p-2}{2}}
\Big\langle \frac{\partial (g^{\alpha \gamma} g^{\beta \lambda} P^\perp)}{\partial p^k_\mu} 
\partial_\mu \phi^k \partial^2_{\alpha \beta} f, \partial^2_{\gamma \lambda} f \Big\rangle \sqrt{\det g}\\
& + & \frac{1}{4} \int_{\Omega} (1+|A|^2)^{\frac{p}{2}}
\frac{\partial \sqrt{\det g}}{\partial p^k_\mu} \partial_\mu \phi^k.
\end{eqnarray*}
In particular if $f(x) = \big(x,u(x)\big)$ where $u \in W^{2,p}(\Omega,\R^{n-2})$, then $f$ is a critical
point of $\mathcal{E}^p$ if and only if $u$ is a weak solution of the system
\begin{equation}
\label{eqgraphsystem}
\partial^2_{\alpha \beta} \big(a^{\alpha \beta}_i(Du,D^2u)\big)
+ \partial_\alpha \big(b^\alpha_i(Du,D^2u)\big) = 0 \quad (1 \leq i \leq n-2),
\end{equation}
where the coefficients are given by
$$
a^{\alpha \beta}_i(Du,D^2u) = (1+|A|^2)^{\frac{p-2}{2}}
\sqrt{\det g}\, g^{\alpha \gamma} g^{\beta \lambda} 
\big(\delta_{ij} - g^{\mu \nu} \partial_\mu u^i \partial_\nu u^j)  \partial^2_{\gamma \lambda} u^j
$$
\quad $b^\alpha_i(Du,D^2 u) = $
\begin{eqnarray*}
&& - \frac{1}{2} (1+|A|^2)^{\frac{p-2}{2}}
\frac{\partial \big(g^{\gamma \mu}g^{\lambda \nu} (\delta_{jk} 
- g^{\sigma \tau} p^j_\sigma p^k_\tau)\big)}{\partial p^i_\alpha}  
\partial^2_{\gamma \lambda} u^j \partial^2_{\mu \nu} u^k \sqrt{\det g}\\ 
&& - \frac{1}{p} (1+|A|^2)^{\frac{p}{2}} \frac{\partial \sqrt{\det g}}{\partial p_\alpha^i}.
\end{eqnarray*}
For $|p| \leq \Lambda$ and $V = (1+|q|^2)^{\frac{1}{2}}$, where $p,q$ are the variables 
corresponding to $Du,D^2 u$, one easily checks the bounds
\begin{eqnarray*}
|D_q a| & \leq & C(\Lambda)\, V^{p-2}\\
|a| + |D_p a| + |D_q b| & \leq & C(\Lambda)\, V^{p-1}\\
|b| + |D_p b| & \leq & C(\Lambda)\,  V^p.
\end{eqnarray*}
Moreover, the system satisfies the ellipticity condition 
$$
\frac{\partial a_i^{\alpha \beta}}{\partial q^j_{\gamma \lambda}}  
\xi^i_{\alpha \beta} \xi^j_{\gamma \lambda} \geq \lambda V^{p-2} |\xi|^2  
\quad \mbox{ where } \lambda = \lambda(\Lambda) > 0,
$$
For the first variation of $\mathcal{W}^p(f)$ one obtains 
\begin{eqnarray*}
\label{Wderivative}
D{\mathcal W}^p(f)\phi & = & \frac{p}{4} \int_\Sigma (1+|H|^2)^{\frac{p-2}{2}}
\langle H, g^{\gamma \lambda} \partial^2_{\gamma \lambda } \phi \rangle\, \sqrt{\det g}\\
\nonumber
& + & \frac{p}{8} \int_\Omega (1+|H|^2)^{\frac{p-2}{2}}
\Big\langle \frac{\partial (g^{\alpha \beta} g^{\gamma \lambda} P^\perp)}{\partial p_\mu^k} \partial_\mu \phi^k
\partial^2_{\alpha \beta} f, \partial^2_{\gamma \lambda} f \Big \rangle \,\sqrt{\det g}\\
& + & \frac{1}{4} \int_{\Omega} (1+|H|^2)^{\frac{p}{2}}
\frac{\partial \sqrt{\det g}}{\partial p^k_\mu} \partial_\mu \phi^k.
\end{eqnarray*}
Putting $L_g \phi = \sqrt{\det g}\, g^{\gamma \lambda} \partial^2_{\gamma \lambda} \phi$ 
the first variation takes the form 
\begin{align}
D{\mathcal W}^p(f)\phi = 
\frac{p}{4} \int_{\Omega} (1+|H|^2)^{\frac{p-2}{2}} \langle H, L_g \phi \rangle
+ \int_\Omega B^\alpha_i(Df,D^2 f) \partial_\alpha \phi^i, \label{EL}
\end{align}
where 
\begin{eqnarray*}
 B^\alpha_i (Df,D^2 f) & = & \frac{p}{8} (1+|H|^2)^{\frac{p-2}{2}} 
\Big\langle \frac{\partial (g^{\gamma \lambda} g^{\mu \nu } P^\perp)}{\partial p_\alpha^i}
\partial^2_{\gamma \lambda} f, \partial^2_{\mu \nu} f \Big\rangle \,\sqrt{\det g}\\
& & + \frac{1}{4} (1+|H|^2)^{\frac{p}{2}} \frac{\partial \sqrt{\det g}}{\partial p^i_\alpha}. 
\end{eqnarray*}
When passing to graphs we have under the assumption $|p| \leq \Lambda$ 
\begin{align}
|B| + |D_p B|  \leq C(\Lambda)\, V^{p-2} |q|^2 \quad \mbox{ and } \quad
|D_q B| \leq C(\Lambda)\, V^{p-2} |q|. \label{EL2}
\end{align}

\section{Regularity of critical points}\label{smoothcrit}
\setcounter{equation}{0}

\subsection{The functional $\mathcal{W}^p$}

For $\Omega \subset \R^2$ and $p > 2$, let $f:\Omega \to \R^n$ be the graph
of a function $u \in W^{2,p}(\Omega,\R^{n-2})$. Recall from \eqref{EL} that 
$f$ is a critical point of $\mathcal{W}^p$ if and only if
\begin{equation}\label{e}
  \int_{\Omega} \langle \mathcal{H},L_g\varphi\rangle
+\int_{\Omega} B_i^\alpha(D u,D^2 u)\partial_\alpha \varphi^i =0 \quad
\mbox{ for all } \varphi\in W^{2,p}_0(\Omega,\R^n).
\end{equation}
Here $\mathcal{H}=(1+|H|^2)^{\frac{p}{2}-1}H$ and the functions $B^i_\alpha$ satisfy 
the bounds \eqref{EL2}. We have the following result. 

\begin{thm}\label{regularityW}
Weak solutions $u\in W^{2,p}(\Omega,\R^{n-2})$ of \eqref{e} are smooth.
\end{thm}

\subsubsection{$W^{2,q}$-regularity}

We start by stating a regularity property for the mean curvature system. 
For a graph of a function  $u \in W^{2,p}(\Omega,\R^{n-2})$, the weak 
mean curvature satisfies for $j = 1,\ldots,n-2$ the formula
\begin{equation}
\label{hsystem}
g^{\alpha \beta} \big(\delta_{ij} - 
g^{\lambda \mu} \partial_\lambda u^i \partial_\mu u^j\big) \partial^2_{\alpha \beta} u^i 
= H^{j+2}. 
\end{equation}
Since $p > 2$ the left hand side may be viewed as a linear operator of the form
$a^{\alpha \beta}_{ij} \partial^2_{\alpha \beta} u^i$, where the coefficients 
are H\"older continuos with exponent $1- \frac{2}{p} > 0$ and the ellipticity 
constant is controlled by the $W^{2,p}$-norm of the function $u$.
In particular, if we know $H \in L^q(\Omega,\R^n)$ for some $q \in (p,\infty)$,
then standard $L^q$ theory yields $u \in W^{2,q}_{loc}(\Omega,\R^{n-2})$ 
together with a local estimate
\begin{align}
\label{estmean}
||u||_{W^{2,q}(\Omega')} \leq C(p,q,\Lambda)\, \big(||H||_{L^q(\Omega)} + 1\big) \quad 
\mbox{ if } \|u\|_{W^{2,p}(\Omega)} \leq \Lambda. 
\end{align}
The dependence on the domains $\Omega' \subset \!\! \subset \Omega$ is not
mentioned explicitely here. 

\begin{lem}\label{Lp}
Let $u\in W^{2,p}(\Omega,\R^{n-2})$ be a weak solution of \eqref{e}.
Then for any $\varphi\in W^{2,p}(\Omega)$ and any test function $\eta \in C^2_0(\Omega)$ 
we have
$$
\int_\Omega \eta  \langle \mathcal{H},L_g\varphi \rangle
\leq C \int_{{\rm spt\,} \eta} (1+|D^2u|^2)^{\frac{p}{2}}(|\varphi|+|D\varphi|),
\quad \mbox{ where } C = C(\|\eta\|_{C^2}).
$$
\end{lem}
\begin{proof}
Expanding 
$$
L_g(\eta\varphi)=
\eta L_g\varphi + \varphi L_g \eta + 2\sqrt{\det g}g^{\alpha \beta}\partial_\alpha \eta \partial_\beta \varphi,
$$
we see by combining with (\ref{e}) and \eqref{EL2} that
\begin{eqnarray*}
\int_{\Omega}\eta \langle \mathcal{H},L_g\varphi\rangle & \leq & 
C\int_\Omega (1+|D^2u|^{2})^{\frac{p}{2}}\big(|D\eta||\varphi|+|\eta| |D \varphi|\big)\\
&& + C \int_\Omega (1+|D^2u|^{2})^{\frac{p-1}{2}}\big(|D^2\eta||\varphi|+|D\eta| |D \varphi|\big).
\end{eqnarray*}
This implies the lemma.
\end{proof}
We are now ready to improve the integrability of $D^2 u$. 
\begin{thm}\label{1st}
Let $u \in W^{2,p}(\Omega,\R^{n-2})$ be a weak solution of \eqref{e} where
$p > 2$. Then $u\in W^{2,q}_{loc}(\Omega,\R^{n-2})$ for any $q \in [p,\infty)$. 
\end{thm}
\begin{proof}
Assume we know already $\|u\|_{W^{2,q}(B_r)} \leq \Lambda$ where $q \geq p$.
For $|H|_A=\min(|H|,A)$ with $A > 0$ and a parameter $\lambda \in (1,q)$, we use
$L^q$-theory to obtain a solution $\varphi\in W^{2,q}\cap W^{1,q}_0(B_r,\R^n)$ 
of the linear equation
$$
L_g \varphi=|H|_A^{\lambda-1}H.
$$
As $1 < \frac{q}{\lambda} < \infty$ the function $\varphi$ satisfies 
$$
||\varphi||_{W^{2,\frac{q}{\lambda}}(B_r)} \le C |||H|_A^{\lambda-1}H||_{L^{\frac{q}{\lambda}}(B_r)} \le C(\Lambda). 
$$
By the Sobolev embedding theorem, we have for $\lambda < \frac{q}{2}$ the estimate
$$
\|\varphi\|_{C^1(B_r)} \leq C(\Lambda),
$$
while for $\frac{q}{2} < \lambda < q$ we get instead
$$
\|\varphi\|_{W^{1,s}(B_r)} \leq C(\Lambda) \quad \mbox{ for } s = \frac{2q}{2\lambda-q} \in [1,\infty).
$$
Now Lemma \ref{Lp} implies that
$$
\int_{B_{\frac{r}{2}}} |H|^p\, |H|_A^{\lambda -1} \leq C \int_{B_r} 
(1+|D^2 u|^2)^{\frac{p}{2}} \big(|\varphi| + |D\varphi| \big)  \leq C(\Lambda),
$$
under the condition that either $1 < \lambda < \frac{q}{2}$, or that $\frac{q}{2} < \lambda < q$ with
$$
\frac{p}{q} + \frac{1}{s} \leq 1 \quad \Leftrightarrow \quad \lambda \leq \frac{3q}{2}-p.
$$
Letting $A \nearrow \infty$ we get $H \in L^{p+\lambda -1}(B_{\frac{r}{2}})$, and then obtain 
$u \in W^{2,p + \lambda-1}(B_{\frac{r}{4}},\R^{n-2})$ from \eqref{estmean}. 
We can now set up an iteration to get $u \in W^{2,q}_{loc}(\Omega)$ for all $q < \infty$.
As initial step we choose $q = p$ and $1 < \lambda  < \frac{p}{2}$, which brings us to 
$q < \frac{3p}{2}-1$. For $p < q < 2p$ we can take $\lambda = \frac{3q}{2}-p$, improving the 
exponent to $\frac{3q}{2}-1$. After finitely many iterations, we arrive at some $q > 2p$. 
Now we continue with $\frac{q}{2} < \lambda = q-p+2 < \frac{3q}{2}-p$ and obtain the desired higher 
integrability. 
\end{proof}

\subsubsection{$W^{1,2}_{loc}$-Regularity of $\mathcal{H}$}

In this subsection we use difference quotient methods in order to show that $\mathcal{H}\in W^{1,2}_{loc}$. For $h>0$, $f:\Omega \to \R^k$ and fixed $\nu\in \{1,2\}$ we define
\[
 f_h(x) = \frac{1}{h} \left(f(x+he_\nu) - f(x)\right)
\]

In the first Lemma we compare the difference quotients of $\mathcal{H}$ with the ones of $H$ and $|H|^2$. 
\begin{lem}\label{lem3}
Let $u$ be as in Theorem \ref{regularityW} and let $\Omega' \subset \! \! \subset \Omega$. Then, for every $1\le q<\infty$ and for all $h>0$ small enough we have for all $x\in \Omega'$
$$|(|H|^2)_h|(x)+|H_h|(x)\leq \Phi(x,h)|\mathcal{H}_h|(x),$$
where
$$\int_{B_r}|\Phi(x,h)|^q dx\le C(q).$$
\end{lem}
\begin{proof}
We have $|\mathcal{H}|^2=(1+|H|^2)^{p-2}|H|^2$ and we get from the mean value theorem
\begin{align}
(|\mathcal{H}|^2)_h(x)=&((1+|H|^2)^{p-2})_h(x)|H|^2(x+he_\nu)+(1+|H|^2)^{p-2}(x)(|H|^2)_h(x)\nonumber \\
   =&(|H|^2)_h(x)[(p-2)(1+\xi(x))^{p-3}|H|^2(x+he_v)+(1+|H|^2)^{p-2}(x)],\label{H}
\end{align}
where $0\le \xi(x) \leq |H|^2(x)+|H|^2(x+he_\nu)$. Hence we have
$$|(|H|^2)_h|(x)\leq |(|\mathcal{H}|^2)_h|(x).$$
On the other hand we calculate
$$(|\mathcal{H}|^2)_h(x)=
(\mathcal{H})_h(x)\mathcal{H}(x+he_\nu)
+(\mathcal{H})_h(x)\mathcal{H}(x).$$
Combining this with Theorem \ref{1st} proves the first estimate.

Another application of the mean value theorem yields
$$\mathcal{H}_h(x)=H_h(x)(1+|H|^2)^{\frac{p}{2}-1}(x+he_\nu)+(\frac{p}{2}-1)(1+\xi(x))^{\frac{p}{2}-2}(|H|^2)_h(x)H(x),$$
where $0\le \xi(x) \leq |H|^2+|H|^2_h$ and hence the second estimate follows from the first one.
\end{proof}

We have the following Corollary.
\begin{cor}\label{d2udq}
Let $u$ be as in Theorem \ref{regularityW} and let $B_r \subset \Omega$. Then, for every $1<s<\infty$ and for all $h>0$ small enough we have 
\begin{align}
\int_{B_r} \eta^{2s}|(D^2 u)_h|^s \le C\int_{B_r}\Phi(x,h)\Big(\eta^{2s}|(\mathcal{H})_h|^s+1\Big), \label{d2udq1}
\end{align}
where $\eta\in C^\infty_c(B_r)$ is a smooth cut-off function and $\Phi$ satisfies
$$\int_{B_r}|\Phi(x,h)|^q dx\le C(q)\ \ \ \forall \,\ 1\le q<\infty.$$
\end{cor}
\begin{proof}
Using \eqref{hsystem} we see that $u_h$ solves
\begin{align*}
a_{ij}^{\alpha \beta}(\cdot+he_\nu) \partial^2_{\alpha \beta} u_h^i 
= H_h^{j+2}- \big(a_{ij}^{\alpha \beta} \big)_h\partial^2_{\alpha \beta} u^i=: \tilde{H}^j(x), 
\end{align*}
where $a_{ij}^{\alpha \beta}=g^{\alpha \beta} \big(\delta_{ij} - 
g^{\lambda \mu} \partial_\lambda u^i \partial_\mu u^j\big)$.

Using Theorem \ref{1st} and Lemma \ref{lem3}, a standard estimate shows that there exists a function $\Phi_1(x,h)$ satisfying 
$$\int_{B_r}|\Phi_1(x,h)|^q dx\le C(q),$$
for all $1<q<\infty$ and all $h>0$ small, such that
\begin{align*}
|\tilde{H}|(x)\le& \Phi_1(x,h)\Big(|\mathcal{H}_h|+1\Big).
\end{align*}
Next we use standard $L^p$-theory in order to get for every $1<s<\infty$
\begin{align*}
\int_{B_r}|D^2(\eta^2 u_h)|^s\le& C\int_{B_r}\Big(\eta^{2s}|\tilde{H}|^s+|Du_h|^s+|u_h|^s\Big)\\
\le& C\int_{B_r}\Phi_2(x,h)\Big(|\mathcal{H}_h|^s+1\Big),
\end{align*}
where 
$\Phi_2(x,h)$ satisfies 
$$\int_{B_r}|\Phi_2(x,h)|^q dx\le C(q),$$
for all $h>0$ small and all $1<q<\infty$.

Since moreover
\begin{align*}
\int_{B_r} \eta^{2s}|(D^2 u)_h|^s \le& C\int_{B_r}\Big(|D^2(\eta^2 u_h)|^s+|D(\eta^2u_h)|^s+|\eta^2u_h|^s\Big)\\
\le& C\int_{B_r}|D^2(\eta^2 u_h)|^s+C\int_{B_r} \Phi_3(x,h),
\end{align*}
for some function $\Phi_3$ with the same properties as $\Phi_1$ and $\Phi_2$, this finishes the proof of the Corollary.
\end{proof}

Now we are in a position to prove that $\mathcal{H} \in W^{1,2}_{loc}(\Omega)$.

\begin{pro}\label{2nd}
Let $u\in W^{2,p}(\Omega,\R^{n-2})$ be as in Theorem \ref{regularityW}. Then we have that $\mathcal{H}\in W^{1,2}_{loc}(\Omega,\R^n)$.
\end{pro}
\begin{proof}
Taking difference quotients of equation \eqref{e} we get
\begin{align}
\partial^2_{\alpha \beta}\big(\sqrt{g} g^{\alpha \beta}\mathcal{H}\big)_h-\partial_\alpha\big(B^\alpha(Du,D^2u)\big)_h=0.\label{eqdiffsystemH}
\end{align}
We abbreviate $U(x) = \big(Du(x),D^2 u(x)\big)$ and we use the fundamental theorem of Calculus to write 
\begin{eqnarray*} 
\left(f \circ U\right)_h (x) & = & \frac{1}{h}\,\Big(f(U(x+he_\nu)) - f(U(x))\Big)\\
& = & \frac{1}{h} \int_0^1 \frac{d}{dt}\, f\big((1-t)U(x) + t\, U(x+h e_\nu)\big) \,dt\\
& = & \int_0^1 Df\big((1-t)U(x) + t\, U(x+he_\nu)\big)\,dt \cdot U_h(x).
\end{eqnarray*}
Using the notation $f^h(x) = \int_0^1 f\big((1-t)U(x) + t U(x+he_\nu)\big)\,dt$ 
we thus get 
$$
\left(f \circ U\right)_h = 
\Big(\frac{\partial f}{\partial q_{\lambda \mu}^j}\Big)^h  \partial^2_{\lambda \mu}  u_h^j
+ \Big(\frac{\partial f}{\partial p_{\lambda}^j}\Big)^h \partial_{\lambda} u_h^j,
$$
and the system (\ref{eqdiffsystemH}) takes the form
\begin{equation}
\label{eqdiffsystemH2}
\partial^2_{\alpha \beta} \Big((\sqrt{g} g^{\alpha \beta})(x+he_\nu)\mathcal{H}_h\Big) +\partial^2_{\alpha \beta} \Big((\sqrt{g} g^{\alpha \beta})_h\mathcal{H}\Big)-
\partial_{\alpha} \left(\tilde{B}^\alpha(Du_h,D^2 u_h)\right)= 0.
\end{equation}
Here the coefficients are given as follows:
\begin{align*} 
\tilde{B}^\alpha(x,z,p,q) =&
\Big(\frac{\partial B^\alpha}{\partial q_{\lambda \mu}^j}\Big)^h(x)  q^j_{\lambda \mu} 
+ \Big(\frac{\partial B^\alpha}{\partial p_{\lambda}^j}\Big)^h(x) p^j_\lambda .
\end{align*}
In order to state bounds for these coefficients, we introduce the abbreviation 
$$
I_{s,h}(x) = 
\int_0^1 \Big(1+ \big|(1-t)D^2 u(x) + t D^2 u(x+h)\big|^2\Big)^{s/2}\,\,dt.
$$
Using \eqref{EL2} we then obtain 
\begin{align}
|\tilde{B}(p,q)| \leq  C \Big(I_{p-1,h}(x) |q| + I_{p,h}(x) |p| \Big).\label{estB}
\end{align}
For $||Du||_{L^\infty}<\infty$ we get 
\begin{align}
|(\sqrt{g} g^{-1})_h\mathcal{H}|(x)\le C (1+|D^2 u|^2)^{\frac{p-1}{2}}(x)|Du_h|(x) \label{estR}
\end{align}
and moreover the operator
\[
\tilde{L}v(x):=(\sqrt{g} g^{\alpha \beta})(x+he_\nu)\partial_{\alpha \beta}^2 v(x)
\]
is strongly elliptic. We let $B_r \subset \Omega$ and by standard $L^p$-theory there exists a solution $\tilde{\varphi}\in W^{2,4}\cap W^{1,4}_0(B_r,\R^n)$ of
\[
\tilde{L} \tilde{\varphi} = \mathcal{H}
\]
satisfying
\[
||\tilde{\varphi}||_{W^{2,4}(B_r)}\le C||\mathcal{H}||_{L^4(B_r)}.
\]
Next we let $\eta\in C^\infty_c(B_r)$ be a smooth cut-off function and we define $\varphi_h=\eta^4 \tilde{\varphi}_h$. A standard computation then shows that
\begin{align*}
 \tilde{L} \varphi_h =& \eta^4 \mathcal{H}_h -(\sqrt{g}g^{\alpha \beta})_h(x+he_\nu)\partial^2_{\alpha \beta}(\eta^4\tilde{\varphi})(x+he_\nu)\\
&+4\eta^2\Big(\eta (\sqrt{g} g^{\alpha \beta})(x+he_\nu)\partial_\alpha \eta \partial_\beta \tilde{\varphi}_h\\
&+\eta (\sqrt{g} g^{\alpha \beta})(x+he_\nu)\partial^2_{\alpha \beta} \eta  \tilde{\varphi}_h+3 (\sqrt{g} g^{\alpha \beta})(x+he_\nu)\partial_{\alpha} \eta \partial_\beta \eta \tilde{\varphi}_h\Big)\\
=& \eta^4 \mathcal{H}_h+\eta^2 R[\tilde{\varphi},D\tilde{\varphi},D^2 \tilde{\varphi}]
\end{align*}
and, by using standard $L^2$-estimates, we conclude
\begin{align}
||\varphi_h||_{W^{2,2}(B_r)}\le& C(r)(||\eta^2 \mathcal{H}_h||_{L^2(B_r)}+||\eta^2 \tilde{\varphi}_h||_{W^{1,2}(B_r)}+||R[\tilde{\varphi},D\tilde{\varphi},D^2 \tilde{\varphi}]||_{L^2(B_r)})\nonumber\\
\le& C(r)\Big(||\eta^2 \mathcal{H}_h||_{L^2(B_r)}+(\int_{B_r}\Phi(x,h)dx)^{\frac12}\Big),\label{test}
\end{align}
where here and in the following we let $\Phi(x,h)$ be a function satisfying
\[
\int_{B_r} |\Phi(x,h)|^qdx\le C(q)\ \ \ \forall \ \ 1\le q<\infty.
\]

Sobolev's embedding theorem then gives for every $q<\infty$
\[
||\varphi||_{W^{1,q}(B_r)}\le C(r)\Big(||\eta^2 \mathcal{H}_h||_{L^2(B_r)}+(\int_{B_r}\Phi(x,h)dx)^{\frac12}\Big).
\]
Using $\varphi_h$ as a test function in \eqref{eqdiffsystemH2} we conclude
\begin{align*}
\int_{B_r}\eta^4 |\mathcal{H}_h|^2\le& C\int_{B_r}\Big(|(\sqrt{g} g^{-1})_h\mathcal{H}||D^2 \varphi_h|+ |\tilde{B}(Du_h,D^2u_h)||D\varphi_h|\\
&+\eta^2|\mathcal{H}_h||R[\tilde{\varphi},D\tilde{\varphi},D^2 \tilde{\varphi}]|\Big)\\
=& I+II+III.
\end{align*}
Next we estimate all three integrals separately. We start with $I$. Combining \eqref{estR}, \eqref{test}, Theorem \ref{1st} and H\"older's inequality we conclude
\begin{align*}
I\le& C||D^2 \varphi_h||_{L^2(B_r)}\big(\int_{B_r}(1+|D^2 u|^2)^{p-1}|Du_h|^2\big)^{\frac12}\\
\le& C(\int_{B_r}\Phi(x,h)dx)^{\frac12}\Big(||\eta^2 \mathcal{H}_h||_{L^2(B_r)}+(\int_{B_r}\Phi(x,h)dx)^{\frac12}\Big)\\
\le&  \delta \int_{B_r}\eta^4 |\mathcal{H}_h|^2+C\int_{B_r}\Phi(x,h)dx.
\end{align*}
Using Corollary \ref{d2udq}, \eqref{estB} and H\"older's inequality we get
\begin{align*}
II\le& C||\tilde{\varphi}_h||_{W^{1,4}(B_r)}(\int_{B_r}I_{p-1,h}^{\frac43}\eta^{\frac83}|D^2u_h|^{\frac43}+I_{p,h}^{\frac43}|Du_h|^{\frac43})^{\frac34}\\
\le& C||\tilde{\varphi}_h||_{W^{1,4}(B_r)}\Big((\int_{B_r}\eta^3 |D^2u_h|^{\frac32})^{\frac23}(\int_{B_r}I_{p-1,h}^{12})^{\frac{1}{12}}+\int_{B_r}\Phi(x,h))^{\frac34}\Big)\\
\le& C||\tilde{\varphi}_h||_{W^{1,4}(B_r)}\Big((\int_{B_r}\eta^3 \Phi(x,h)|\mathcal{H}_h|^{\frac32}+\Phi(x,h))^{\frac23}\\
&\cdot (\int_{B_r}\Phi(x,h))^{\frac{1}{12}}+(\int_{B_r}\Phi(x,h))^{\frac34}\Big).
\end{align*}
Using the fact that $\tilde{\varphi} \in W^{2,4}(B_r)$ we get for $h>0$ small enough
\begin{align*}
II\le \delta \int_{B_r}\eta^4 |\mathcal{H}_h|^2+\int_{B_r}\Phi(x,h).
\end{align*}
Finally we get
\begin{align*}
III\le& \delta \int_{B_r}\eta^4 |\mathcal{H}_h|^2+\int_{B_r}|R[\tilde{\varphi},D\tilde{\varphi},D^2 \tilde{\varphi}]|^2\\
\le& \delta \int_{B_r}\eta^4 |\mathcal{H}_h|^2+\int_{B_r}\Phi(x,h).
\end{align*}
Combining all these estimates yields for $\delta$ small enough
\begin{align*}
\int_{B_r}\eta^4 |\mathcal{H}_h|^2\le \int_{B_r}\Phi(x,h)
\end{align*}
and therefore we can let $h\rightarrow 0$ in order to conclude that $\mathcal{H} \in W^{1,2}_{loc}(\Omega,\R^{n})$ with
\begin{align*}
\int_K\eta^4 |D\mathcal{H}|^2\le C
\end{align*}
for all $K\subset \subset \Omega$.
\end{proof}

\begin{cor}\label{3s}
Let $u\in W^{2,p}(\Omega,\R^{n-2})$ be as in Theorem \ref{regularityW}. Then we have that $H\in W^{1,s}_{loc}(\Omega,\R^{n})$ and $u\in W^{3,s}_{loc}(\Omega,\R^{n-2})$ for all $s<2$.
\end{cor}
\begin{proof}
Combining Theorem \ref{1st} and Proposition \ref{2nd} we get that
\[
|\mathcal{H}|^2 \in W^{1,s}_{loc}(\Omega)
\]
for all $s<2$. Using the formulas
\begin{align}
D |\mathcal{H}|^2=D |H|^2\Big((1+|H|^2)^{p-2}+(p-2)|H|^2(1+|H|^2)^{p-3}\Big)\label{1}
\end{align}
and
\begin{align}
(1+|H|^2)^{\frac{p-2}{2}}D H=D \mathcal{H}-\frac{p-2}{2}(1+|H|^2)^{\frac{p-4}{2}}H D |H|^2, \label{2}
\end{align}
we conclude that
\begin{align*}
|DH|\le& C(|D\mathcal{H}|+|D|\mathcal{H}|^2|).
\end{align*}
Therefore we get that
\[
H \in W^{1,s}_{loc}(\Omega,\R^{n})
\]
for all $s<2$. Arguing as at the beginning of this subsection we conclude that
\[
 u\in W^{3,s}_{loc}(\Omega,\R^{n-2})
\]
for all $s<2$.
\end{proof}

\subsubsection{Higher regularity}

In this last subsection we show the higher regularity for solutions of \eqref{e}. We start by showing that $\mathcal{H} \in W^{1,2+\gamma}(\Omega,\R^{n})$ for some $0<\gamma<\frac12$. 

In order to see this, we let $B_r\subset \Omega$ and we let $\varphi_1\in W^{2,\frac{2}{1+\gamma}}\cap W^{1,\frac{2}{1+\gamma}}_0(B_r,\R^{n})$ be the solution of
\[
\tilde{L}\varphi_1=|\mathcal{H}_h|^\gamma \mathcal{H}_h
\]
satisfying
\begin{align*}
||\varphi_1||_{W^{2,\frac{2}{1+\gamma}}(B_r)}\le C||\mathcal{H}_h||_{L^2(B_r)}^{1+\gamma}.
\end{align*}
Sobolev's embedding theorem yields
\begin{align*}
||\varphi_1||_{L^\infty(B_r)}+||D\varphi_1||_{L^{\frac{2}{\gamma}}(B_r)}\le  C||\mathcal{H}_h||_{L^2(B_r)}^{1+\gamma}.
\end{align*}

Next we let $\eta\in C^\infty_c(B_r)$ and we define $\tilde{\varphi_1}=\eta^4 \varphi_1$. We conclude that
\begin{align*}
\tilde{L}\tilde{\varphi_1}=&\eta^4 |\mathcal{H}_h|^\gamma \mathcal{H}_h+8\eta^3 (\sqrt{g}g^{\alpha \beta})(x+he_\nu)\partial_\alpha \eta \partial_\beta \varphi_1\\
&+4\eta^2 \varphi_1 (\sqrt{g}g^{\alpha \beta})(x+he_\nu)(3 \partial_\alpha \eta \partial_\beta \eta+ \eta \partial^2_{\alpha \beta} \eta)\\
=:& \eta^4 |\mathcal{H}_h|^\gamma \mathcal{H}_h+\eta^2L[\varphi_1,D\varphi_1].
\end{align*}
Moreover we have that
\begin{align*}
||\tilde{\varphi_1}||_{L^\infty(B_r)}+||D\tilde{\varphi_1}||_{L^{\frac{2}{\gamma}}(B_r)}+||D^2\tilde{\varphi}_1||_{L^{\frac{2}{1+\gamma}}(B_r)}\le C||\mathcal{H}_h||_{L^2(B_r)}^{1+\gamma}.
\end{align*}

Now we use $\tilde{\varphi}_1$ as a test function in equation \eqref{eqdiffsystemH2} and we conclude
\begin{align*}
\int_{B_r}\eta^4 |\mathcal{H}_h|^{2+\gamma}\le& C\int_{B_r}\Big(|(\sqrt{g} g^{-1})_h\mathcal{H}||D^2 \tilde{\varphi}_1|+ |\tilde{B}(Du_h,D^2u_h)||D\tilde{\varphi}_1|\\
&+\eta^2|\mathcal{H}_h||L[\varphi_1,D\varphi_1]|\Big)\\
=& i)+ii)+iii).
\end{align*}
As before we estimate the three integrals separately. Using \eqref{estR} and H\"older's inequality we get for $h$ small enough
\begin{align*}
i)\le& C||D^2\tilde{\varphi}_1||_{L^{\frac{2}{1+\gamma}}(B_r)}||(1+|D^2 u|^2)^{\frac{p-1}{2}}|Du_h|||_{L^{\frac{2}{1-\gamma}}(B_r)}\\
\le& C.
\end{align*}
Corollary \ref{3s} and \eqref{estB} imply for $h$ small enough
\begin{align*}
ii)\le& C||D\tilde{\varphi_1}||_{L^{\frac{2}{\gamma}}(B_r)}\Big(||D^2u_h||_{L^{\frac{4}{3}}(B_r)}||I_{p-1,h}||_{L^{\frac{4}{1-2\gamma}}(B_r)}\\
&+||Du_h||_{L^{\frac{4}{3}}(B_r)}||I_{p,h}||_{L^{\frac{4}{1-2\gamma}}(B_r)}\Big)\\
\le& C.
\end{align*}
Finally we have
\begin{align*}
iii)\le& C||\mathcal{H}_h||_{L^2(B_r)}||\varphi_1||_{W^{1,2}(B_r)}\le C.
\end{align*}
Combining thes estimates we conclude
\begin{align*}
\int_{B_{\frac{r}{2}}} |\mathcal{H}_h|^{2+\gamma}\le C
\end{align*}
and therefore we get that
\[
\mathcal{H} \in W^{1,2+\gamma}_{loc}(\Omega,\R^{n}).
\]
In particular this implies that $\mathcal{H}\in L^\infty_{loc}(\Omega,\R^{n-2})$ and hence
\[
|\mathcal{H}|^2 \in W^{1,2+\gamma}_{loc}(\Omega).
\]
Now we can argue as in the proof of Corollary \ref{3s} in order to get
\begin{align*}
u\in W^{3,2+\gamma}_{loc}(\Omega,\R^{n-2}).
\end{align*}
By the Sobolev embedding theorem this gives
\[
u\in C^{2,\beta}_{loc}(\Omega,\R^{n-2}).
\] 
for some $\beta>0$. The smoothness of solutions of equation \eqref{e} now follows from classical Schauder theory.

\subsection{The functional $\mathcal{E}^p(f)$}

Here we consider for $p > 2$ weak solutions $u\in W^{2,p}(\Omega, \R^m)$ of elliptic 
systems in two independent variables of the form 
\begin{equation}
\label{eqsystem}
\partial^2_{\alpha \beta} \left[a^i_{\alpha \beta}(\cdot,u,Du,D^2u)\right] + 
\partial_{\alpha} \left[b^i_{\alpha}(\cdot,u,Du,D^2u)\right]+
c^i(\cdot,u,Du,D^2u) = 0.
\end{equation}
We assume that $a,b,c$ are $C^1$ functions satisfying the following ellipticity and 
growth conditions at all points $(x,z,p,q)$, for $V(x,z,p,q) = (1+|q|^2)^{1/2}$
and for constants $\lambda > 0$, $C < \infty$: 
\begin{equation}
\label{eqellipticity}
\frac{\partial a^i_{\alpha \beta}}{\partial q^j_{\lambda \mu}} 
\xi^i_{\alpha \beta} \xi^j_{\lambda \mu} 
\geq \lambda \,V^{p-2} |\xi|^2,
\end{equation}
$$
|a_q| \leq C\,V^{p-2}\nonumber
$$
\begin{equation} 
\label{eqgrowth}
|a|+|a_x|+|a_z|+|a_p|+|b_q|+|c_q| \leq C\,V^{p-1}
\end{equation}
$$
|b|+|b_x| + |b_z| + |b_p| + |c| + |c_x| + |c_z| + |c_p| \leq C\,V^p.
$$
As noted in section \ref{euler}, the graph function of a critical point 
for $\mathcal{E}^p$ satisfies a system of the required form, with 
suitable bounds (\ref{eqellipticity}) and (\ref{eqgrowth}). Therefore
Theorem \ref{regularity} is a consequence of the following Proposition
\ref{generalsys} and standard higher regularity theory, for which we refer
to \cite{morrey66}.

\begin{pro}\label{generalsys}
Let $u\in W^{2,p}(\Omega,\R^m)$ be a weak solution of \eqref{eqsystem}, where 
$p>2$ and $\Omega\subset \R^2$, and assume that \eqref{eqellipticity} and 
\eqref{eqgrowth} hold. Then $u$ belongs to $C^{2,\alpha}_{loc}(\Omega,\R^m)$
for some $\alpha > 0$. 
\end{pro}
\begin{rem}
A related regularity result, for functionals where the integrand satisfies 
a more general (anisotropic) ellipticity condition but depends only on the 
second derivatives, was proved in \cite{bildhauer06}. A crucial ingredient 
both in \cite{bildhauer06} and in our paper is the Gehring type lemma from
Bildhauer, Fuchs and Zhong \cite{bildhauer05}.
\end{rem}

\subsubsection{Growth estimate}
In a first step we show a growth estimate for the $L^p$-norm of the second derivatives of weak solutions of the system \eqref{eqsystem}.
\begin{lem}\label{morrey}
Let $p>2$. There exist $r_0>0$, $\beta>0$ and $C>0$ such that if $u\in W^{2,p}(\Omega, \R^m)$ is a weak solution 
of the elliptic system \eqref{eqsystem} which satisfies \eqref{eqellipticity} and \eqref{eqgrowth}, then we have 
for every $B_{2r}(x) \subset \Omega$ with $r < r_0$ that
\begin{align}
\int_{B_r(x)} V^p \le c \Big(\frac{r}{r_0}\Big)^\beta. \label{morrey1}
\end{align}
\end{lem}
\begin{proof}
Let $r_0>0$. Since $u\in W^{2,p}\cap L^\infty(\Omega, \R^m)$ we get from the Sobolev embedding theorem that $u\in C^{1,\gamma}(\Omega, \R^m)$ for some $\gamma>0$. Now we choose $x_0\in \Omega$ and we let $0<2r<\min\{2r_0,\text{dist}(x_0,\partial \Omega)\}$. Moreover we let $A_r= B_{2r} \backslash B_r(x_0)$ and $\varphi \in C^\infty_c (B_{2r}(x_0))$ be a smooth cut-off function which satisfies
\begin{align}
0\le \varphi \le 1,\ \ \  \varphi=1 \ \ \text{in}\ \  B_r(x_0),\ \ \  ||\nabla^j \varphi||_{L^\infty} \le cr^{-j} \ \ \forall j\in \N. \label{test1}
\end{align}
Finally we define the linear function $l_r$ by
\begin{align*}
l_r(x) = \frac{1}{|A_r|}\int_{A_r} u+ (x-x_0)\cdot \frac{1}{|A_r|}\int_{A_r} D u.
\end{align*}
>From this definition it easily follows that we have the estimates
\begin{align}
||u-l_r||_{L^\infty(B_{2r})} +r||D (u-l_r)||_{L^\infty(B_{2r})} &\le Cr^{1+\gamma}\ \ \ \text{and} \label{morrey2}\\
||u-l_r||_{L^p(A_r)} +r||D (u-l_r)||_{L^p(A_r)} &\le C r^2||D^2 u||_{L^p(A_r)} . \label{morrey3}
\end{align}
Now we choose $\varphi^4 (u-l_r)$ as a test function in the weak form of the system \eqref{eqsystem} and we get
\begin{align*}
|\int_\Omega \varphi^4 a_{\alpha \beta}^i \partial_{\alpha \beta}^2 u^i| \le& C \int_\Omega \varphi^4 (|b||D (u-l_r)|+|c||u-l_r|)\\
&+Cr^{-1}\int_{A_r} \varphi^3 (|a||D(u-l_r)|+|b||u-l_r|)\\
&+Cr^{-2} \int_{A_r} \varphi^2 |a||u-l_r|\\
=& I+II+III.
\end{align*}
Using the ellipticity assumption and the bound $|a(x,z,p,0)|\le C$ (which follows from \eqref{eqgrowth}) we estimate
\begin{align*}
q_{\alpha \beta}^{i} a_{\alpha \beta}^{i}(x,y,z,q)&=q_{\alpha \beta}^{i} a_{\alpha \beta}^{i}(x,y,z,0)+q_{\alpha \beta}^{i}q_{\gamma \delta}^j\int_0^1 a_{\alpha \beta,q_{\gamma \delta}^j}^i(x,y,z,tq)dt\\
&\ge q_{\alpha \beta}^{i} a_{\alpha \beta}^{i}(x,y,z,0)+\lambda |q|^2\int_0^1 (1+t^2|q|^2)^{\frac{p-2}{2}}dt\\
&\ge \tilde \lambda V^p-C
\end{align*}
where $\tilde \lambda>0$ is some number. Next we use \eqref{eqgrowth}, H\"olders inequality and the estimates \eqref{morrey2} and \eqref{morrey3} to obtain
\begin{align*}
I\le& Cr^\gamma \int_\Omega \varphi^4 V^p, \\
II \le&  Cr^\gamma\int_{A_r} V^p+Cr^{-1}||V||_{L^p(A_r)}^{p-1}||D(u-l_r)||_{L^p(A_r)}\le C\int_{A_r} V^p\ \ \ \text{and} \\
III \le& Cr^{-2}(\int_{A_r} V^p)^{\frac{p-1}{p}}||u-l_r||_{L^p(A_r)}\le C\int_{A_r} V^p.
\end{align*}
Combining all these estimates and choosing $r_0$ small enough we conclude that there exists a constant $C_1>0$ such that
\begin{align*}
\int_{B_r} V^p \le C_1 \int_{A_r} V^p +Cr^2.
\end{align*}
Adding $C_1 \int_{B_r} V^p$ to both sides of this estimate we get
\begin{align*}
\int_{B_r} V^p\le \frac{C_1}{C_1+1} \int_{B_{2r}} V^p +Cr^2.
\end{align*}
The estimate \eqref{morrey1} now follows from a standard iteration argument.
\end{proof}
\subsubsection{Difference quotient estimates}\label{diffquo}
In a second step we use the difference quotient method to show that every weak solution $u \in W^{2,p}\cap L^\infty(\Omega,\R^m)$ of \eqref{eqsystem} is in $W^{3,2}(B_{\frac{r_0}{4}}(x_0),\R^m)$ and that moreover $U=V^{\frac{p}{2}}\in W^{1,2}(B_{\frac{r_0}{4}}(x_0))$, where $x_0 \in \Omega$ and $r_0$ is as in Lemma \ref{morrey} (in the following we allow the constants to depend on $r_0$). We follow closely the methods developed in \cite{morrey66} and \cite{simon93}.

In the following we use the abbreviation $U(x) = \big(x,u(x),Du(x),D^2 u(x)\big)$.  
Applying the difference quotient to the equation \eqref{eqsystem} and interchanging with the derivatives yields
\begin{equation}
\label{eqdiffsystem}
\partial^2_{\alpha \beta} \left[a^i_{\alpha \beta} \circ U\right]_h +
\partial_{\alpha} \left[b^i_{\alpha} \circ U\right]_h +
\left[c^i \circ U\right]_h = 0. 
\end{equation}
We may use the fundamental theorem of Calculus to write 
\begin{eqnarray*} 
\left[f \circ U\right]_h (x) & = & \frac{1}{h}\,\Big(f(U(x+he_\nu)) - f(U(x))\Big)\\
& = & \frac{1}{h} \int_0^1 \frac{d}{dt}\, f\big((1-t)U(x) + t\, U(x+h e_\nu)\big) \,dt\\
& = & \int_0^1 Df\big((1-t)U(x) + t\, U(x+he_\nu)\big)\,dt \cdot U_h(x).
\end{eqnarray*}
Using the notation $f^h(x) = \int_0^1 f\big((1-t)U(x) + t U(x+he_\nu)\big)\,dt$ 
we thus get 
$$
\left[f \circ U\right]_h = 
\Big(\frac{\partial f}{\partial q_{\lambda \mu}^j}\Big)^h  \partial^2_{\lambda \mu}  u_h^j
+ \Big(\frac{\partial f}{\partial p_{\lambda}^j}\Big)^h \partial_{\lambda} u_h^j
+ \Big(\frac{\partial f}{\partial z^j}\Big)^h u_h^j + \Big(\frac{\partial f}{\partial x_\nu}\Big)^h,
$$
and the system (\ref{eqdiffsystem}) takes the form
\begin{equation}
\label{eqdiffsystem2}
\partial^2_{\alpha \beta} \left[\tilde{a}^i_{\alpha \beta}(\cdot,u_h,Du_h,D^2 u_h)\right] +
\partial_{\alpha} \left[\tilde{b}^i_{\alpha}(\cdot,u_h,Du_h,D^2 u_h)\right]+
\tilde{c}^i(\cdot,u_h,D u_h,D^2 u_h) = 0.
\end{equation}
Here the coefficients are given as follows:
\begin{align*} 
\tilde{a}^i_{\alpha \beta}(x,z,p,q) =& 
\Big(\frac{\partial a^i_{\alpha \beta}}{\partial q_{\lambda \mu}^j}\Big)^h(x)  q^j_{\lambda \mu} 
+ \Big(\frac{\partial a^i_{\alpha \beta}}{\partial p_{\lambda}^j}\Big)^h(x) p^j_\lambda 
+ \Big(\frac{\partial a^{i}_{\alpha \beta}}{\partial z^j}\Big)^h(x) z^j \\
&+ \Big(\frac{\partial a^i_{\alpha \beta}}{\partial x_\nu}\Big)^h(x),\\
\tilde{b}^i_{\alpha}(x,z,p,q) =&
\Big(\frac{\partial b^i_{\alpha}}{\partial q_{\lambda \mu}^j}\Big)^h(x)  q^j_{\lambda \mu} 
+ \Big(\frac{\partial b^i_{\alpha}}{\partial p_{\lambda}^j}\Big)^h(x) p^j_\lambda 
+ \Big(\frac{\partial b^{i}_{\alpha}}{\partial z^j}\Big)^h(x) z^j 
+ \Big(\frac{\partial b^i_{\alpha}}{\partial x_\nu}\Big)^h(x),\\
\tilde{c}^i(x,z,p,q) =&
\Big(\frac{\partial c^i}{\partial q_{\lambda \mu}^j}\Big)^h(x)  q^j_{\lambda \mu} 
+ \Big(\frac{\partial c^i}{\partial p_{\lambda}^j}\Big)^h(x) p^j_\lambda  
+ \Big(\frac{\partial c^i}{\partial z^j}\Big)^h(x) z^j 
+ \Big(\frac{\partial c^i}{\partial x_\nu}\Big)^h(x).
\end{align*}
In order to state bounds for these coefficients, we introduce the abbreviation 
$$
I_{s,h}(x) = 
\int_0^1 \Big(1+ \big|(1-t)D^2 u(x) + t D^2 u(x+h)\big|^2\Big)^{s/2}\,\,dt.
$$
Using (\ref{eqellipticity}) and (\ref{eqgrowth}) we then obtain 
\begin{eqnarray*}
|\tilde{a}(x,z,p,q)| & \leq & C \Big(I_{p-2,h}(x) |q| + I_{p-1,h}(x) (|p| + |z| + 1)\Big)\\
|\tilde{b}(x,z,p,q)|+|\tilde{c}(x,z,p,q)| & \leq & C \Big(I_{p-1,h}(x) |q| + I_{p,h}(x) (|p| + |z| + 1)\Big).
\end{eqnarray*}
As above we define the linear function $l_{h,r}=:l_h$ by
\begin{align*}
l_h(x) = \frac{1}{|A_r|}\int_{A_r} u_h+ (x-x_0)\cdot \frac{1}{|A_r|}\int_{A_r} D u_h.
\end{align*}
Using again the test function $\varphi^4 (u_h -l_h)$, we infer
\begin{align*}
\int \tilde{a}^i_{\alpha \beta} \partial^2_{\alpha \beta} u_h^i \varphi^4 \leq&
C \int \Big(I_{p-1,h} |D^2 u_h| + I_{p,h} (|Du_h| + |u_h| + 1)\Big)\\
&\cdot \Big(|D(u_h-l_h)| + |u_h-l_h|\Big)\,\varphi^4\\
&+ \frac{C}{r} \int \big(I_{p-2,h} |D(u_h-l_h)| + I_{p-1,h} |u_h-l_h|\big)\, |D^2 u_h|\,\varphi^3\\
&+ \frac{C}{r} \int \big(I_{p-1,h} |D(u_h-l_h)| + 
I_{p,h} |u_h-l_h|\big) \\
&\cdot \big(|Du_h| + |u_h| + 1\big)\,\varphi^3\\
&+ \frac{C}{r^2} \int \Big(I_{p-2,h} |D^2 u_h| + I_{p-1,h} (|Du_h| + |u_h| + 1)\Big)\\
&\cdot |u_h-l_h|\,\varphi^2.
\end{align*}
On the other hand we have the ellipticity condition (using \eqref{eqellipticity})
\begin{eqnarray*}
\label{eqellipticitytilde}
\tilde{a}^i_{\alpha \beta}(x,z,p,q) q^i_{\alpha \beta} & = & 
\int_0^1 \frac{\partial a^i_{\alpha \beta}}{\partial q_{\lambda \mu}^j}\big((1-t)U(x) + t U(x+h)\big)
q^i_{\alpha \beta} q^j_{\lambda \mu}\,dt\\
&+& \int_0^1 \frac{\partial a^i_{\alpha \beta}}{\partial p_{\lambda}^j}\big((1-t)U(x) + t U(x+h)\big)
p^j_\lambda q^i_{\alpha \beta}\,dt\\
&+& \int_0^1  \frac{\partial a^{i}_{\alpha \beta}}{\partial z^j}\big((1-t)U(x) + t U(x+h)\big)
z^j q^i_{\alpha \beta}\,dt\\
&+& \int_0^1 \frac{\partial a^i_{\alpha \beta}}{\partial x_\nu}\big((1-t)U(x) + t U(x+h)\big)
q^i_{\alpha \beta}\,dt\\
& \geq & C \lambda \,I_{p-2,h} |q|^2 - C I_{p-1,h}\,|q| \big(|p| + |z| + 1\big).
\end{eqnarray*}
Combining the two inequalities we arrive at
\begin{align*}
\int I_{p-2,h} |D^2 u_h|^2 \varphi^4  \leq& 
C \int I_{p-1,h} |D^2 u_h| 
\Big(|D(u_h-l_h)| \varphi^4 \\
&+ |u_h-l_h| (\varphi^4 + \frac{\varphi^3}{r}) + 
(|Du_h| + |u_h| +1)\varphi^4\Big)\\ 
&+ C \int I_{p-2,h} |D^2 u_h| 
\Big(|D(u_h-l_h)| \frac{\varphi^3}{r} + |u_h-l_h| \frac{\varphi^2}{r^2}\Big)\\
&+ C \int I_{p,h} \big(|Du_h| + |u_h| + 1\big)
\Big(|D(u_h-l_h)| \varphi^4  \\
&+ |u_h-l_h| (\varphi^4 + \frac{\varphi^3}{r})\Big)\\
&+ C \int I_{p-1,h} \big(|Du_h| + |u_h| + 1\big)
\Big(|D(u_h-l_h)| \frac{\varphi^3}{r} \\
&+ |u_h-l_h| \frac{\varphi^2}{r^2}\Big).
\end{align*}
Using $I_{p-1,h} \leq I_{p-2,h}^{1/2} I_{p,h}^{1/2}$ and absorbing the 
second derivatives of $u_h$ yields 
\begin{align*}
\int I_{p-2,h} |D^2 u_h|^2 \varphi^4  \leq&
C \int I_{p,h} \Big(|D(u_h-l_h)|^2 + |u_h-l_h|^2 + |Du_h|^2 \\
&+ |u_h|^2 + 1\Big) \varphi^4+ C \int I_{p,h} |u_h-l_h|^2 \frac{\varphi^2}{r^2}\\
&+C \int I_{p-2,h} \Big(|D(u_h-l_h)|^2 \frac{\chi_{A_r}}{r^2}  + |u_h-l_h|^2 \frac{\chi_{A_{r}}}{r^4}\Big)\\
=& I + II + III.
\end{align*}

Before continuing we need to recall the following Lemma which is essentially due to Morrey \cite{morrey66}, Lemma $5.4.2$. In the form stated here it can be found in \cite{simon93}. 
\begin{lem}\label{mor}
Let $r>0$ and let the function $q\ge 0$ be such that
\begin{align*}
\int_{B_{s}(x) \cap B_r} q\le cs^\gamma
\end{align*}
for all $B_s(x) \subset B_r$. Then for every $\eps>0$ there exists $C_\eps>0$ such that
\begin{align*}
\int_{B_r} q|v|^2 \le \eps r^\gamma \int_{B_r} |Dv|^2+C_\eps r^{\gamma-2}\int_{B_r} |v|^2.
\end{align*}
\end{lem}

Next we use this Lemma in order to estimate the terms $I$-$III$ from above. By the definition of $I_{p,h}$ and Lemma \ref{morrey} we get that 
\begin{align*}
\int_{B_s(x)}I_{p,h} \le Cs^\beta
\end{align*}
for all $B_s(x)\subset B_{2r}\subset B_{r_0}$ and some $0<\beta<1$. Hence we can apply Lemma \ref{mor} and H\"older's respectively Poincar\'e's inequality to estimate
\begin{align*}
II\le& Cr^{\beta-2} ||D(u_h-l_h)||_{L^{2}(B_{2r})}^2.
\end{align*}
Using the same argument we get 
\begin{align*}
I\le& C r^\beta \int_{B_{2r}} \varphi^4(|D^2 u_h|^2+|D(u_h-l_h)|^2)\\
&+Cr^{\beta-2}\int_{B_{2r}}  (|D u_h|^2+|D(u_h-l_h)|^2+|u_h|^2+|u_h-l_h|^2+I_{p,h}).
\end{align*}
Inserting these two estimates into the above estimate for $\int I_{p-2,h} |D^2 u_h|^2 \varphi^4$ we conclude that
\begin{align}
\int I_{p-2,h} |D^2 u_h|^2 \varphi^4\le& C r^\beta \int_{B_{2r}} \varphi^4(|D^2 u_h|^2+|D(u_h-l_h)|^2)+Cr^{\beta-2}\int_{B_{2r}} \Big(|D u_h|^2 \nonumber \\
&+|D(u_h-l_h)|^2+|u_h|^2+|u_h-l_h|^2+I_{p,h}\Big) \nonumber \\
&+C \int I_{p-2,h} \Big(|D(u_h-l_h)|^2 \frac{\chi_{A_r}}{r^2}  + |u_h-l_h|^2 \frac{\chi_{A_{r}}}{r^4}\Big).\label{important}
\end{align}
Next we use H\"older`s and Poincar\'e's inequality to get
\begin{align*}
\int I_{p-2,h} \Big(&|D(u_h-l_h)|^2 \frac{\chi_{A_r}}{r^2}  + |u_h-l_h|^2 \frac{\chi_{A_{r}}}{r^4}\Big) \\
\le& C||I_{p-2,h}||_{L^{\frac{p}{p-2}}(B_{2r})}(r^{-4}||u_h-l_h||_{L^{p}(A_r)}^2+r^{-2}||D(u_h-l_h)||_{L^{p}(A_r)})\\
\le& Cr^{-2} ||I_{p-2,h}||_{L^{\frac{p}{p-2}}(B_{2r})}||D(u_h-l_h)||_{L^{p}(A_r)}^2.
\end{align*}

Since $u\in W^{2,p}(\Omega)$ we know from Theorem $3.6.8$ in \cite{morrey66} that
\begin{align*}
I_{s,h} \rightarrow V^s \ \ \ \text{in}\ \ L^{\frac{p}{s}}\ \ \ \forall\ \ 1\le s\le p.
\end{align*}

We combine all the above estimates to get ($l_h\rightarrow \mint_{A_r}\partial_\nu u+(x-x_0)\mint_{A_r}\partial_\nu Du$)
\begin{align*}
\int \varphi^4 I_{p-2,h} |D^2 u_h|^2 \le& C(1+r^{-2})\int_{B_{2r}}V^p+Cr^{\beta-2}\int_{B_{2r}}  |D^2 u|^2 +Cr^{\beta}. 
\end{align*}
In particular this estimate is true for $r=\frac{r_0}{4}$ and therefore we can let $h\rightarrow 0$ to conclude 
\begin{align}
\int_{B_{\frac{r_0}{8}}} V^{p-2} |D^3 u|^2 \le& C (1+r_0^{-2})\int_{B_{\frac{r_0}{4}}}V^p+cr_0^{\beta-2}\int_{B_{\frac{r_0}{4}}}  |D^2 u|^2+Cr_0^{\beta}.\label{w32}
\end{align}
Hence we have that $u\in W^{3,2}(B_{\frac{r_0}{8}})$ and by the Sobolev embedding theorem this implies $u\in W^{2,q}_{\text{loc}}(B_{\frac{r_0}{8}})$ for all $q<\infty$. Moreover the above estimate yields that $U\in W^{1,2}(B_{\frac{r_0}{8}})$. Altogether this shows that we can improve Lemma \ref{morrey} to get the estimate
\begin{align}
\int_{B_r} V^p\le cr^{2-\delta}\label{morrey3ab}
\end{align}
for all $r\le \frac{r_0}{16}$  and all $\delta>0$.

\subsubsection{Higher regularity}

In order to obtain the higher regularity for weak solutions of \eqref{eqsystem} we need to recall the following estimate from the previous subsection:
\begin{align*}
\int \varphi^4 \tilde{a}^i_{\alpha \beta} \partial^2_{\alpha \beta} u^i_h \le& C\int \varphi^4 \Big(I_{p-1,h}|D^2 u_h|+I_{p,h}(|Du_h|+|u_h|+1)\Big) \Big( |D(u_h-l_h)|\\
&+|u_h-l_h|\Big)\\
&+ \frac{C}{r}\int \varphi^3  \Big( I_{p-2,h}|D(u_h-l_h)|+I_{p-1,h}|u_h-l_h|\Big)|D^2 u_h|\\
&+ \frac{C}{r}\int \varphi^3  \Big( I_{p-1,h}|D(u_h-l_h)|+I_{p,h}|u_h-l_h|\Big)\Big(|D u_h|\\
&+|u_h|+1\Big)\\
&+ \frac{C}{r^2}\int \varphi^2 \Big(I_{p-2,h}|D^2 u_h|+I_{p-1,h}(|Du_h|+|u_h|+1)\Big) |u_h-l_h|\\
=&I+\ldots+IV.
\end{align*}
This time we choose $l_h$ such that
\begin{align*}
\int_{B_{2r}} (u_h-l_h)=&0\ \ \ \text{and}\\
\int_{B_{2r}} (Du_h-Dl_h)=&0.
\end{align*}
Because of \eqref{morrey3ab} and the strong convergence $I_{s,h}\rightarrow V^s$ in $L^{\frac{p}{s}}$ we have for every $r\le \frac{r_0}{16}$, every $h$ small enough and every $\delta>0$
\begin{align}
\int_{B_r} I_{p,h} \le cr^{2-\delta}.\label{morrey3a}
\end{align}
Now we estimate again each term seperately. We start with $I$. By Young`s inequality we get
\begin{align*}
I\le& \eps \int \varphi^4 I_{p-2,h} |D^2 u_h|^2 +C\int \varphi^4 I_{p,h}\Big(|Du_h|^2+|u_h|^2+1\Big)\\
&+ C\int \varphi^4 I_{p,h}\Big(|D(u_h-l_h)|^2+|u_h-l_h|^2\Big)
\end{align*}
and we continue to estimate the last term with the help of Lemma \ref{mor}, \eqref{morrey3a} and Poincar\'e`s inequality by
\begin{align*}
\int \varphi^4 I_{p,h}\Big(|D(u_h-l_h)|^2+|u_h-l_h|^2\Big)\le& Cr^{2-\delta}\int_{B_{2r}} (|D^2(u_h-l_h)|^2+|D(u_h-l_h)|^2)\\
&+Cr^{-\delta}\int_{B_{2r}} (|D(u_h-l_h)|^2+|u_h-l_h|^2)\\
\le&  Cr^{2-\delta} \int_{B_{2r}} |D^2 u_h|^2.
\end{align*}
Next we estimate
\begin{align*}
II\le& \eps \int \varphi^4 I_{p-2,h} |D^2 u_h|^2+\frac{C}{r^2} \int_{B_{2r}} I_{p,h} |u_h-l_h|^2\\
&+\frac{C}{r}(\int_{B_{2r}} I_{p-2,h}^{\frac43}|D^2 u_h|^{\frac43})^{\frac34}(\int_{B_{2r}} |D(u_h-l_h)|^4)^{\frac14}.
\end{align*}
The second term can be estimated as above to yield
\begin{align*}
\frac{C}{r^2} \int_{B_{2r}} I_{p,h} |u_h-l_h|^2\le Cr^{2-\delta} \int_{B_{2r}} |D^2 u_h|^2
\end{align*}
and for the third term we use the Sobolev-Poincar\'e inequality to get
\begin{align*}
\frac{C}{r}(\int_{B_{2r}} I_{p-2,h}^{\frac43}|D^2 u_h|^{\frac43})^{\frac34}(\int_{B_{2r}} |D(u_h-l_h)|^4)^{\frac14}\le \frac{C}{r} (\int_{B_{2r}} I_{p-2,h}^{\frac43}|D^2 u_h|^{\frac43})^{\frac32}.
\end{align*}
$III$ can be estimated by 
\begin{align*}
III\le& \frac{\gamma}{r^2} \int_{B_{2r}} (|D(u_h-l_h)|^2+I_{p,h}|u_h-l_h|^2)\\
&+C_\gamma \int_{B_{2r}}(I_{2p-2,h}+I_{p,h})(|Du_h|^2+|u_h|^2+1)\\
\le&  C(\gamma+r^{2-\delta}) \int_{B_{2r}} |D^2 u_h|^2+C_\gamma \int_{B_{2r}}(I_{2p-2,h}+I_{p,h})(|Du_h|^2+|u_h|^2+1).
\end{align*}
Finally, using some of the estimates from above, the last term is estimated as follows
\begin{align*}
IV\le& \frac{C}{r^2}(\int_{B_{2r}} I_{p-2,h}^{\frac43}|D^2 u_h|^{\frac43})^{\frac34}(\int_{B_{2r}} |u_h-l_h|^4)^{\frac14}\\
&+ \frac{\gamma}{r^4}\int_{B_{2r}} |u_h-l_h|^2 +C_\gamma \int_{B_{2r}}I_{2p-2,h}(|Du_h|^2+|u_h|^2+1)\\
\le& \frac{C}{r} (\int_{B_r} I_{p-2,h}^{\frac43}|D^2 u_h|^{\frac43})^{\frac32}+C\gamma \int_{B_{2r}} |D^2 u_h|^2\\
&+ C_\gamma \int_{B_{2r}}I_{2p-2,h}(|Du_h|^2+|u_h|^2+1).
\end{align*}
We also note that we have the ellipticity estimate 
\begin{align*}
\int \varphi^4 I_{p-2,h} |D^2u_h|^2 \le C\int \varphi^4 \tilde{a}^i_{\alpha \beta} \partial^2_{\alpha \beta} u^i_h+C\int \varphi^4 I_{p,h}(|Du_h|^2+|u_h|^2+1).
\end{align*}
Combining all these estimates we get
\begin{align*}
\int_{B_{r}}I_{p-2,h} |D^2u_h|^2\le& \frac{C}{r} (\int_{B_{2r}} I_{p-2,h}^{\frac43}|D^2 u_h|^{\frac43})^{\frac32}+C(\gamma+r^{2-\delta}) \int_{B_{2r}} |D^2 u_h|^2\\
&+C_\gamma \int_{B_{2r}}(I_{2p-2,h}+I_{p,h})(|Du_h|^2+|u_h|^2+1).
\end{align*} 
Since $u\in W^{3,2}(B_r)$ and $U\in W^{1,2}(B_r)$ for $r\le \frac{r_0}{16}$ we can let $h\rightarrow 0$ and get
\begin{align*}
\int_{B_{r}}V^{p-2} |D^3 u|^2\le& \frac{C}{r} (\int_{B_{2r}} V^{\frac43(p-2)}|D^3 u|^{\frac43})^{\frac32}+C(\gamma+r_0^{2-\delta}) \int_{B_{2r}} V^{p-2}|D^3 u|^2\\
&+C_\gamma \int_{B_{2r}}V^{2p}.
\end{align*} 
Defining $f=(V^{\frac{p-2}{2}}|D^3 u|)^{4/3}$, $g=V^{\frac{2(p-2)}{3}}$, $h=V^{\frac{4p}{3}}$ and $d=\frac32$ we conclude the inequality
\begin{align}
(\mint_{B_{r}} f^d)^{\frac{1}{d}}\le C\mint_{B_{2r}} fg +C(\gamma+r_0^{2-\delta})^{\frac1d}(\mint_{B_{2r}} f^d)^{\frac{1}{d}} +C_\gamma (\mint_{B_{2r}} h^d)^{\frac1d}\label{bfza}
\end{align}
for all balls $B_r \subset B_{\frac{r_0}{16}}$. Next we need the following Gehring type Lemma, which slightly generalizes Lemma $1.2$ of Bildhauer, Fuchs and Zhong \cite{bildhauer05} (see also Theorem $1.1$ in \cite{faraco05}).
\begin{lem}\label{bfz}
Let $d>1$, $\beta>0$ be two constants. There exists $\eps_0>0$ such that for all all $\eps<\eps_0$ and all non-negative functions $f,g,h:\Omega \subset \R^n \rightarrow \R$ satisfying
\[
f,h\in L^d_{\text{loc}}(\Omega),\ \ \ e^{\beta g^d}\in  L^1_{\text{loc}}(\Omega)
\]
and (for some constant $C>0$)
\begin{align} 
(\mint_B f^d)^{\frac{1}{d}}\le C \mint_{2B}fg +\eps(\mint_{2B}f^d)^{\frac{1}{d}}+(\mint_{2B} h^d)^{\frac{1}{d}}\label{bfz1}
\end{align}
for all balls $B=B_r(x)$ with $B_{2r}(x)\subset \subset \Omega$. Then there exists $c_0=c_0(n,d,C)>0$ such that if 
\[
h^d \text{log}^{c_0 \beta}(e+h)\in L^1_{\text{loc}}(\Omega),
\]
then the same is true for $f$. Moreover, for all balls $B$ as above we have
\begin{align}
\mint_B f^d \text{log}^{c_0 \beta}(e+\frac{f}{||f||_{L^d(2B)}})\le& c(\mint_{2B}e^{\beta g^d})(\mint_{2B}f^d)\nonumber \\
&+c\mint_{2B}h^d \text{log}^{c_0 \beta}(e+\frac{f}{||f||_{d,2B}}),\label{bfz2}
\end{align}
where $c=c(n,d,\beta,C)>0$ and $||f||_{d,2B}=(\mint_{2B}f^d)^{\frac{1}{d}}$.
\end{lem}
\begin{proof}
The proof is very similar to the one of Lemma $1.2$ in \cite{bildhauer05} and therefore we only comment on the differences.

We define $B_0=2B$ and we assume without loss of generality that
\[
\int_{B_0}f^d=1.
\]
Next we define the functions $d(x)=\text{dist}(x,\R^n\backslash B_0)$ and
\begin{align*}
\tilde{f}(x)=&d(x)^{\frac{n}{d}}f,\\
\tilde{h}(x)=&d(x)^{\frac{n}{d}}h \ \ \ \text{and}\\
w(x)=& \chi_{B_0}(x),
\end{align*}
where $\chi_{B_0}$ is the characteristic function of $B_0$. As in \cite{bildhauer05} it is now easy to see that because of \eqref{bfz1} these new functions satisfy
\begin{align*}
(\mint_{B} \tilde{f}^d)^{\frac{1}{d}}\le& C\mint_{2B} \tilde{f}g +C\eps(\mint_{2B} \tilde{f}^d)^{\frac{1}{d}} +C (\mint_{2B} \tilde{h}^d)^{\frac1d}\\
&+C(\mint_{2B} w)^{\frac1d}
\end{align*}
and now this inequality is true for all balls $B \subset \R^n$. Hence, by taking the supremum over all radii, we get (here $M(f)$ denotes the maximal function of $f$)
\begin{align*}
M(\tilde{f}^d)^{\frac{1}{d}}\le& CM(\tilde{f}g)+C\eps M(\tilde{f}^d)^{\frac1d}+C M(\tilde{h}^d)^{\frac1d}+CM(w)^{\frac1d}.
\end{align*}
For $\eps_0$ small enough we therefore have
\begin{align*}
M(\tilde{f}^d)^{\frac{1}{d}}\le& CM(\tilde{f}g)+C M(\tilde{h}^d)^{\frac1d}+CM(w)^{\frac1d}
\end{align*}
and with the help of this inequality we can copy the rest of the argument of the proof of Lemma $1.2$ in \cite{bildhauer05} to finish the proof.
\end{proof}

Now we want to apply this Lemma to our estimate \eqref{bfza}. From the previous subsection we know that 
\begin{align*}
f^d=& V^{p-2}|D^3 u|^2\in L^1_{loc}(B_{\frac{r_0}{16}}) \ \ \ \text{and}\\
h^d=& V^{2p} \in L^1_{loc}(B_{\frac{r_0}{16}}).
\end{align*}
Hence it remains to check that
\begin{align}
e^{\beta g^d}=e^{\beta V^{p-2}} \in L^1_{loc}(B_{\frac{r_0}{16}}) \label{gexp}
\end{align}
for some constant $\beta>0$. We actually claim that this is true for all $\beta>0$. In order to see this we note that by \eqref{w32} we have that
\begin{align*}
\int_{B_{\frac{r_0}{8}}}|D (V^{\frac{p}{2}})|^2 \le c_1(r_0).
\end{align*}
Next we let $\eta\in C^\infty_c(B_{\frac{r_0}{8}})$ be a cut-off function such that $0\le \eta\le 1$, $\eta (x)\equiv 1$ for all $x\in B_{\frac{r_0}{16}}$ and $||D \eta||_{L^\infty(B_{\frac{r_0}{8}})}\le cr_0^{-1}$. Defining $v=\eta V^{\frac{p}{2}}$ we get that
\[
\int_{B_{\frac{r_0}{8}}}|D v|^2\le cr_0^{-2}\int_{B_{\frac{r_0}{8}}}V^p+c\int_{B_{\frac{r_0}{8}}}|D V^{\frac{p}{2}}|^2 \le c_2(r_0).
\]
Hence we see that $u=\frac{v}{\sqrt{c_2(r_0)}}\in H^1_0(B_{\frac{r_0}{8}})$ and
\[
 \int_{B_{\frac{r_0}{8}}}|D u|^2\le 1.
\]
Therefore, by the Moser-Trudinger inequality (see \cite{trudinger67}), there exist constants $\beta_0>0$ and $C=C(r_0)>0$ such that
\[
\int_{B_{\frac{r_0}{16}}} e^{\beta_0 V^p}\le \int_{B_{\frac{r_0}{8}}} e^{c_2(r_0)\beta_0 u^2} \le C.
\]
In particular this implies with the help of Young's inequality that for every $\beta>0$ 
\begin{align*}
\int_{B_{\frac{r_0}{16}}} e^{\beta V^{p-2}}\le c(\beta,\beta_0)\int_{B_{\frac{r_0}{16}}} e^{\beta_0 V^p}\le C(r_0,\beta,\beta_0).
\end{align*}
Since we also have that 
\[
h^d\text{log}^\alpha (e+h)=V^{2p}\text{log}^\alpha (e+V^{\frac{4p}{3}})\in L^1_{loc}(B_{\frac{r_0}{16}})
\]
for every $\alpha>0$ we get from Lemma \ref{bfz} that 
\[
f^d\text{log}^\alpha (e+f)\in L^1_{loc}(B_{\frac{r_0}{16}})
\]
for every $\alpha>0$. Hence we obtain that
\[
|D^3 u|^2 \text{log}^\alpha (e+|D^3 u|)\in L^1_{loc}(B_{\frac{r_0}{16}})
\]
for every $\alpha>0$. In particular this is true for $\alpha>1$ and therefore we can apply Corollary $4.6$ and Example $4.18(iv)$ of \cite{edmunds97} (see also Example $5.3$ in \cite{kauhanen99} for a different proof of this result) in order to conclude that 
\begin{align*}
u\in C^2(B_{\frac{r_0}{32}}).
\end{align*}
In particular this implies that 
\begin{align}
\int_{B_r} V^p \le cr^2\label{vr2}
\end{align}
for all $r\le \frac{r_0}{32}$. 

In order to show the H\"older continuity of $D^2 u$ we go back to \eqref{important}
and we estimate the last term with the help Lemma \ref{mor} and \eqref{vr2} by
\begin{align*}
\int I_{p-2,h} \Big(&|D(u_h-l_h)|^2 \frac{\chi_{A_r}}{r^2}  + |u_h-l_h|^2 \frac{\chi_{A_{r}}}{r^4}\Big) \\
\le& C \int_{A_r} |D^2u_h|^2.
\end{align*}
Inserting this estimate into \eqref{important}, letting $h\rightarrow 0$ and using \eqref{vr2} we conclude for every $r\le \frac{r_0}{32}$
\begin{align}
\int_{B_r} |D^3 u|^2 \le C \int_{A_r} |D^3u|^2+Cr^\beta.\label{hole}
\end{align}
\begin{rem}
This estimate is sufficient for our purposes but by repeating all the estimates from subsection \ref{diffquo} and replacing every application of Lemma \ref{morrey} by \eqref{vr2} one can actually improve this inequality in the sense that the term $r^\beta$ on the right hand side can be replaced by $r^2$. 
\end{rem}
The H\"older continuity of $D^2 u$ now follows from \eqref{hole} by another hole-filling argument. This finishes the proof of Proposition \ref{generalsys}.

\section{Compactness results and existence of minimizers}  
\setcounter{equation}{0}
\subsection{Compactness results}
We start by quoting the fundamental compactness theorem of J. Langer (see also \cite{breuning11}). 
\begin{thm}\cite{langer85}\label{p0}
Let $\Sigma$ be a closed surface and $p > 2$. Assume that $f_k \in W^{2,p}_{{\rm im}}(\Sigma,\R^n)$
is a sequence satisfying $0\in f_k(\Sigma)$ for all $k\in \N$ and
\begin{equation}
\label{langer0}
\mathcal{E}^p(f_k) \le C.
\end{equation}
After replacing $f_k$ by $f_k \circ \varphi_k$ for suitable diffeomorphisms 
$\varphi_k \in C^\infty(\Sigma,\Sigma)$ and passing to a subsequence, the 
$f_k$ converge weakly in $W^{2,p}(\Sigma,\R^n)$ to an $f \in W^{2,p}_{{\rm im}}(\Sigma,\R^n)$. 
In particular, the convergence is in $C^{1,\beta}(\Sigma,\R^n)$ for any 
$\beta < 1- \frac{2}{p}$, and
\begin{equation}
\label{langer1}
\mathcal{E}^p(f) \le \liminf_{k \to \infty} \mathcal{E}^p(f_k).
\end{equation}
\end{thm}

In this section we prove Theorem \ref{comp}, which replaces the 
$\mathcal{E}^p$ bound in Langer's theorem by a bound only for 
$\mathcal{W}^p$, under the additional assumption that the Willmore 
energy is bounded below $8\pi$. Before entering the proof we include 
two remarks about the statement. 

\begin{rem}\label{genusbound}
One can allow sequences $f_k:\Sigma_k \to \R^n$ in Theorem \ref{comp}, where 
$\Sigma_k$ are arbitrary closed oriented surfaces. In fact, a bound on the genus 
follows from the condition $\liminf_{k\to \infty} \mathcal{W}(f_k) < 8\pi$
by a result of Kuwert, Li \& Sch\"atzle \cite{KLS09}.
\end{rem} 

\begin{rem}\label{optimal}
Connecting two round spheres by a shrinking catenoid neck yields a sequence of smoothly 
embedded surfaces with bounded $\mathcal{W}^p$-energy and Willmore energy less than $8\pi$. 
As the convergence is not in $C^1$, this shows that the assumption on the Willmore
energy in Theorem \ref{comp} cannot be weakened. Similar constructions are also 
possible for higher genus, see K\"uhnel \& Pinkall \cite{kuehnel86} and Simon
\cite{simon93}.
\end{rem}

To prove Theorem \ref{comp} we need the following area ratio bounds, which 
are immediate consequences of Simon's monotonicity identity \cite{simon93}.

\begin{lem}\label{sim}
Let $f:\Sigma \rightarrow \R^n$ be an embedded closed surface. Then
\begin{align}
\sigma^{-2}|\Sigma\cap B_\sigma|\leq C\,\mathcal{W}(f) \quad \mbox{ for all } \sigma > 0.\label{density}
\end{align}
Moreover for any $p > 2$ we have the estimate 
\begin{equation}\label{es}
\sigma^{-2}|\Sigma \cap B_\sigma|\leq \frac{1}{4} \mathcal{W}(f)+C\sigma^{\frac{p-2}{p}}
\quad \mbox{ for all } \sigma > 0,
\end{equation}
where the constant $C$ depends on $\mathcal{W}^p(f)$.
\end{lem}
\begin{proof}
By equation $(1.2)$ in \cite{simon93}, we have for $0 < \sigma < \infty$ the inequality 
\begin{eqnarray*}
\sigma^{-2}|\Sigma \cap B_\sigma| & \leq & \rho^{-2}|\Sigma \cap B_\rho| + \frac{1}{4} \mathcal{W}(f)\\
&& + \frac{1}{2}\int_{\Sigma \cap B_\rho} \rho^{-2} \langle x,H \rangle\,d\mu 
   - \frac{1}{2} \int_{\Sigma \cap B_\sigma} \sigma^{-2} \langle x,H \rangle\,d\mu.
\end{eqnarray*}
Letting $\rho\rightarrow \infty$ we conclude for every $\sigma > 0$
\begin{equation}\label{Simmon}
\sigma^{-2}|\Sigma \cap B_\sigma| \leq \frac{1}{4}\mathcal{W}(f) 
+ \frac{1}{2\sigma}\int_{\Sigma \cap B_\sigma} |H|d\mu.
\end{equation}
The estimates now follow from an application of H\"older's inequality.
\end{proof}

As second ingredient, we need the following lemma yielding an $L^p$ estimate 
for the prescribed mean curvature system.

\begin{lem}\label{2.44}
Let $u\in W^{2,p}(B_\varrho,\mathbb{R}^k)$, where $B_\varrho = \{x \in \mathbb{R}^2: |x| < \varrho\}$
and $0<\varrho<\infty$, $p \in (1,\infty)$, be a solution of the system
\begin{align*}
a^{\alpha \beta}_{ij} \partial^2_{\alpha \beta} u^i = \varphi_j \quad
\mbox{ for }j = 1,\ldots,k.
\end{align*}
There is an $\eps_0 = \eps_0(p) > 0$ such that if
$$
|a^{\alpha \beta}_{ij}(x)-\delta^{\alpha \beta} \delta_{ij}| \le \eps_0  \quad \mbox{ for all }x \in B_\varrho,
$$
then for some $C = C(p) < \infty$ we have the estimate
\begin{align*}
\|D^2 u\|_{L^p(B_\frac{\varrho}{2})} \leq C\big(\|\varphi\|_{L^p(B_\varrho)} + \varrho^{-1} \|Du\|_{L^p(B_\varrho)}\big).
\end{align*}
\end{lem}
\begin{proof} We may assume that $\varrho=1$ and that $u$ has mean value zero on $B_1$. For
$\eta\in C^\infty_0(B_1)$ satisfying $\eta = 1$ on $B_\frac{1}{2}$
and $\eta = 0$ in $B_1\setminus B_{\frac{3}{4}}$, we calculate
\begin{align*}
  a^{\alpha \beta}_{ij} \partial^2_{\alpha \beta}(\eta u^i) = &
\eta \varphi^j + a^{\alpha \beta}_{ij} (\partial_{\alpha \beta}^2 \eta) u^i
+ a^{\alpha \beta}_{ij}( \partial_\alpha \eta\, \partial_\beta u^i + \partial_\beta \eta\, \partial_\alpha u^i).
\end{align*}
Hence we have
\begin{align*}
\Delta(\eta u^j)=&
(\delta^{\alpha \beta} \delta_{ij} - a^{\alpha \beta}_{ij}) \partial^2_{\alpha \beta }(\eta u^i) + \eta \varphi^j\\
& + a^{\alpha \beta}_{ij} (\partial_{\alpha \beta}^2 \eta) u^i 
+ a^{\alpha \beta}_{ij}( \partial_\alpha \eta\, \partial_\beta u^i + \partial_\beta \eta\, \partial_\alpha u^i).
\end{align*}
>From standard $L^p$-estimates and the Poincar\'{e} inequality we obtain
$$
\|D^2 (\eta u)\|_{L^p(B_1)} \leq C \eps_0 \|D^2(\eta u)\|_{L^p(B_1)}
          + C\big(\|\varphi\|_{L^p(B_1)} +\|Du\|_{L^p(B_1)}\big),
$$
for a constant $C = C(p) < \infty$. This shows the desired result.
\end{proof}


\begin{proof}[Proof of Theorem \ref{comp}]
Let $f_k:\Sigma \rightarrow \R^n$ be a sequence as in the theorem. 
For each $q\in \Sigma$, we let $r_k(q) > 0$ be the maximal radius on which 
$f_k$ is represented as a graph over the tangent plane at $q$. We denote by
$u_{k,q}:B_{r_k(q)} \rightarrow \R^{n-2}$ the corresponding graph function, 
obtained by choosing a suitable rigid motion. In particular
$$
u_{k,q}(0) = 0 \quad \mbox{ and } \quad Du_{k,q}(0) = 0.
$$
For $\eps > 0$ we define 
$r_k(q,\varepsilon) = \sup\{r \in (0,r_k(q)]: \|D u_{k,q}\|_{C^0(B_r)}< \varepsilon\}$
and
\begin{align*}
r_k = \inf_{q \in \Sigma} r_k(q,\varepsilon). 
\end{align*}
By compactness, the infimum is attained at some point $q_k\in \Sigma$ 
and we have $r_k > 0$. We will show by contradiction that 
\begin{align}
\liminf\limits_{k\rightarrow \infty} r_k>0. \label{radius}
\end{align}
Assuming that $r_k \rightarrow 0$ we rescale by putting 
$$
\tilde{f}_k:\Sigma \to \R^n,\,\tilde{f}_k(p) = \frac{1}{r_k} \big(f_k(p)-f_k(q_k)\big).
$$
Clearly, the $\tilde{f}_k$ have local graph representations
\begin{align*}
\tilde{u}_{k,q}:B_{r_k(q)/r_k} \to \R^{n-2},\, \tilde{u}_{k,q}(x) = \frac{1}{r_k} u_{k,q}(r_k x),
\end{align*}
where $\tilde{u}_{k,q}(0) = 0$ and $D\tilde{u}_{k,q}(0) = 0$, and 
$$
\|\tilde{u}_{k,q}\|_{C^0(B_1)} + \|D \tilde{u}_{k,q}\|_{C^0(B_1)} \leq C \varepsilon.
$$
>From the bound $\mathcal{W}^p(f_k)\le C$ we further infer that 
$$
\int_\Sigma |H_{\tilde{f}_k}|^{p}\,d\mu_{\tilde{f}_k} =
r_k^{p-2}\,\int_\Sigma |H_{f_k}|^p\,d\mu_{f_k} \le C\,r_k^{p-2} \to 0.
$$
The prescribed mean curvature system \eqref{hsystem} for the $u_k$ fulfills 
the assumption of Lemma \ref{2.44}, if $\eps = \eps(p) > 0$ is
sufficiently small. Therefore we get the  $L^p$ estimate 
\begin{align*}
\|D^2 \tilde{u}_{k,q}\|_{L^p(B_{\frac{1}{2}})} \leq C \,
\big(\|H_{\tilde{f}_k}\|_{L^p(B_1)} + \|D \tilde{u}_{k,q}\|_{L^p(B_1)} \big) \leq C.
\end{align*}
Moreover the monotonicity formula \eqref{density} yields for $B_R = B_R(0) \subset \R^n$
\begin{align*}
\mu_{\tilde{f}_k}(B_R) \leq C\,R^2 \quad \mbox{ for any } R \in (0,\infty).
\end{align*}
We now apply a localized version of Langer's theorem from \cite{breuning11}
to obtain a proper immersion $f_0:\Sigma_0 \to \R^n$, such that 
the $\tilde{f}_k$ converge to $f_0$ locally in $C^{1,\beta}$ up to 
diffeomorphisms. Weak lower semicontinuity of $\mathcal{W}^p$ implies
\begin{align*}
H_{f_0} = 0.
\end{align*}
The Gau{\ss}-Bonnet theorem yields 
$$
\int_\Sigma |A_{f_k}|^2 \,d\mu_{f_k} = 4 \mathcal{W}(f_k) - 4\pi \chi(\Sigma) \leq C,
$$
and thus we get further
\begin{align*}
\int_{\Sigma_0} |A_{f_0}|^2\,d\mu_{f_0} \leq \liminf_{k \rightarrow \infty}
\int_{\Sigma} |A_{\tilde{f_k}}|^2\, d\mu_{\tilde{f_k}} \leq C.
\end{align*}
Summarizing, we have that $f_0:\Sigma_0 \to \R^n$ is a properly immersed minimal surface 
with finite total curvature. By results of Chern \& Osserman, see \cite{CO67}, the 
immersion $f_0$ admits a conformal reparametrisation on a compact surface with 
finitely many punctures, corresponding to the ends. Moreover, each end has a well-defined 
tangent plane and multiplicity. Now the monotonicity formula from 
\eqref{Simmon} implies
$$
\frac{\mu_{\tilde{f}_k}(B_R)}{\pi R^2} = \frac{\mu_{f_k}(B_{r_kR})}{\pi(Rr_k)^2}
\leq \frac{1}{4} \mathcal{W}(f_k)+ C (Rr_k)^{\frac{p-2}{p}} 
\quad \mbox{ for all }R > 0.
$$
Letting $k\rightarrow \infty$ and then $R\to \infty$ we conclude that 
\begin{align}
\limsup_{R \rightarrow \infty} \frac{\mu_{f_0}(B_R)}{\pi R^2} < 2.\label{catenoid}
\end{align}
This means that $f_0$ has just one simple end, and is in fact a plane. 
We now argue that the Gau{\ss} map converges to a constant locally uniformly
on $\Sigma_0 = \R^2$, contradicting the definition of $r_k$.
More precisely, from the compactness theorem in \cite{breuning11} we know that 
$\tilde{f}_k \circ \phi_k \to f_0$ locally in $C^1$ and moreover
$$
\|\tilde{f}_k \circ \phi_k - f_0\|_{C^0(U_k)} \to 0,
$$
where the $U_k \subset \R^2$ are open sets with
$U_1 \subset U_2 \subset \ldots$ and $\R^2 = \bigcup_{k=1}^\infty U_k$,
such that $\phi_k:U_k \to \tilde{f}_k^{-1}(B_k(0))$ is diffeomorphic.
Now $\tilde{f}_k(q_k) = 0$ by construction, therefore there
exists a $p_k \in U_k$ with $\phi_k (p_k) = q_k$. In
particular we have
$$
f_0(p_k) = f_0(p_k) - (\tilde{f}_k \circ \phi_k)(p_k) \to 0.
$$
Since $f_0$ is proper, we get $p_k \to p \in \R^2$ after
passing to a subsequence. Now by the indirect assumption, there
exist $x_k \in B_1(0)$ such that for all $k$
$$
|D\tilde{u}_{k,q_k}(x_k) - D\tilde{u}_{k,q_k}(0)| \geq \frac{\eps}{2} > 0.
$$
Denote the corresponding point by $q_k' \in \Sigma$. Then
$|\tilde{f}_k(q_k')| \leq C$, and hence there are points $p_k' \in \R^2$
with $\phi_k(p_k') = q_k'$. This implies
$$
|f(p_k')|
\leq |f(p_k') - (\tilde{f}_k \circ \phi_k)(p_k')| + |\tilde{f}_k(q_k')|
\leq C.
$$
Using again that $f_0$ is proper, we conclude after passing to a further
subsequence that $p_k' \to p' \in \R^2$. But now 
$
T_p f 
= \lim_{k \to \infty} T_{p_k} (\tilde{f}_k \circ \phi_k)
= \lim_{k \to \infty} T_{q_k} \tilde{f}_k
$, and analogously 
$T_{p'} f = \lim_{k \to \infty} T_{q_k'} \tilde{f}_k$.
From the indirect assumption, we obtain $T_p f \neq T_{p'}f$,
contradicting the fact that $f$ parametrizes a plane.\\
\\
Given \eqref{radius} we may finally use Lemma \ref{2.44} with
$\varrho: = \inf_{k \in \N} r_k > 0$ to get
\begin{align*}
\int_{B_{\frac{\varrho}{2}}} |D^2 u_{k,q}|^p \leq C \quad \mbox{ for all }q \in \Sigma,\,k \in \N.
\end{align*}
The global mass bound and a standard covering argument then imply that
\[
\mathcal{E}^p(f_k)\le C \quad \mbox{ for all } k \in \N.
\]
The desired conclusion now follows from Theorem \ref{p0}. 
\end{proof}

\subsection{Existence of minimizers}

Combining Theorem \ref{p0} with our regularity from Theorem \ref{regularity} we immediately get

\begin{thm}\label{p1}
For a closed surface $\Sigma$ and $p > 2$, denote by $\alpha^n_{\Sigma}(p)$ 
the infimum of the energy $\mathcal{E}^p$ among all smooth immersions from 
$\Sigma$ into $\R^n$. Then $\alpha^n_{\Sigma}(p)$ is attained by a 
smooth immersion $f:\Sigma \to \R^n$.
\end{thm}
\begin{proof} 
Using mollification it is easy to see that 
$$
\alpha^n_{\Sigma}(p) = \inf \{\mathcal{E}^p(f): f \in W^{2,p}_{{\rm im}}(\Sigma,\R^n)\}.
$$
Thus the limiting map $f \in W^{2,p}_{{\rm im}}(\Sigma,\R^n)$ of a minimizing sequence
obtained from Theorem \ref{p0} is a critical point of $\mathcal{E}^p$, and hence 
smooth after composing with a diffeomorphism by Theorem \ref{regularity}. 
\end{proof}

For any fixed immersion $f:\Sigma \to \R^n$ we have
$$
\lim_{q \to p} \alpha^n_\Sigma(q) \leq \lim_{q \to p} \mathcal{E}^q(f)
= \mathcal{E}^p(f).
$$
Taking the infimum with respect to $f$ shows that the function 
$\alpha^n_\Sigma:[2,\infty) \to \R$ is upper semicontinuous. In 
particular, the function is continuous from the right since it 
is nondecreasing. For $\lambda > 0$ and  $f:\Sigma \to \R^n$ 
fixed we also note 
$$
\alpha^n_\Sigma(2) \leq \mathcal{E}^2(\lambda f)
= \frac{\lambda^2}{4}\mu_g(\Sigma) + \frac{1}{4} \int_{\Sigma} |A|^2\,d\mu_g
\leq \frac{\lambda^2}{4}\mu_g(\Sigma) + \mathcal{E}^2(f).
$$
Letting $\lambda \searrow 0$ and taking the infimum with respect to $f$ shows
\begin{equation}
\label{scaling2energy}
\alpha^n_\Sigma(2) = \inf_{f:\Sigma \to \R^n} \frac{1}{4} \int_{\Sigma} |A|^2\,d\mu_g.
\end{equation}
Recall that by the Gau{\ss} equation and the Gau{\ss}-Bonnet theorem 
\begin{eqnarray}
\label{eqgauss}
\mathcal{W}(f) = \frac{1}{4} \int_{\Sigma} |A|^2\,d\mu_g + \pi \chi(\Sigma).
\end{eqnarray}
The infimum of the Willmore energy among immersions of $\Sigma$ into $\R^n$
satisfies $\beta^n_\Sigma < 8\pi$  \cite{B-K}. Thus for $p > 2$ close 
to $2$, we conclude for a minimizer $f$ of $\mathcal{E}^p$ that
$$
{\mathcal W}(f) \leq  \mathcal{E}^p(f) +  \pi \chi(\Sigma)
= \alpha^n_\Sigma(p) +  \pi \chi(\Sigma)
< 8\pi.
$$
In particular, these minimizers are embedded by the
Li-Yau inequality \cite{li82}.\\
\\

Next we define the number $\beta^n_{\Sigma}(p)$ as 
the infimum of the energy $\mathcal{W}^p$ among all smooth immersions from 
$\Sigma$ into $\R^n$. Repeating the previous discussion with $\beta^n_{\Sigma}$ instead of $\alpha^n_\Sigma$ 
we conclude that for every sequence of immersions $f_k:\Sigma\to \R^n$ with $\mathcal{W}^p(f_k)\to \beta^n_{\Sigma}$ and $p-2$ small enough we have
$$
{\mathcal W}(f_k) \leq  \mathcal{W}^p(f_k)
\to \beta^n_\Sigma(p) 
< 8\pi.
$$
Combing this estimate with Theorem \ref{comp}, Theorem \ref{regularityW} and arguing as in the proof of Theorem \ref{p1} we get
\begin{thm}\label{p1W}
There exists $2<p_0<\infty$ such that for every closed surface $\Sigma$ and every $2<p<p_0$ the number $\beta^n_{\Sigma}(p)$ is attained by a 
smooth immersion $f:\Sigma \to \R^n$.
\end{thm}
  
The numbers $\alpha^n_\Sigma(p)$ and $\beta^n_{\Sigma}(p)$ depend only on the topological type
of $\Sigma$. This can be refined by minimizing in regular homotopy
classes of immersions $f:\Sigma \to \R^n$. 
Theorem \ref{p1} and Theorem \ref{p1W} extend without any difficulties.

\section{Palais-Smale condition} \label{sectionps}
\setcounter{equation}{0}
Here we show that for $p > 2$ the functionals $\mathcal{E}^p$ resp. $\mathcal{W}^p$ satisfy the Palais-Smale condition resp. a modified Palais-Smale condition, 
up to the action of diffeomorphisms on $\Sigma$. For $f \in W^{2,p}_{{\rm im}}(\Sigma,\R^n)$ 
and any $V \in W^{2,p}(\Sigma,\R^n)$ we define the norm
\begin{align*}
\|V\|_{W^{2,p}_f(\Sigma)} =
\Big(\int_\Sigma \big(|\nabla (DV)|_{g}^{p}+|DV|_{g}^{p} + |V|^{p}\big)\,d\mu_g\Big)^{\frac{1}{p}},
\end{align*}
where $g \in W^{1,p}(T^{0,2}\Sigma)$ is the metric induced by $f$ and $\nabla$ 
denotes its Levi-Civit\`{a} connection, with Christoffel symbols locally in $L^p$.
In particular, the norm is well-defined. Now put
\begin{align*}
\|D\mathcal{E}^p(f)\|_f = \sup \{D\mathcal{E}^p(f)V: 
V \in W^{2,p}(\Sigma,\R^n),\,\|V\|_{W^{2,p}_f(\Sigma)} \leq 1\}, 
\end{align*}
resp.
\begin{align*}
\|D\mathcal{W}^p(f)\|_f = \sup \{D\mathcal{W}^p(f)V: 
V \in W^{2,p}(\Sigma,\R^n),\,\|V\|_{W^{2,p}_f(\Sigma)} \leq 1\}, 
\end{align*}
For any diffeomorphism $\varphi \in W^{2,p}(\Sigma,\Sigma)$ we have that 
$(f \circ \varphi)^\ast g_{{\rm euc}} = \varphi^\ast (f^\ast g_{{\rm euc}})$, which implies  
$\|V\circ\varphi\|_{f\circ\varphi}=\|V\|_f$ and therefore 
\begin{align}
\label{psequivariance}
\|D\mathcal{E}^p(f\circ\varphi)\|_{f\circ\varphi} = \|D\mathcal{E}^p(f)\|_f,
\end{align}
resp.
\begin{align}
\label{psequivariancea}
\|D\mathcal{W}^p(f\circ\varphi)\|_{f\circ\varphi} = \|D\mathcal{W}^p(f)\|_f,
\end{align}
Now we can formulate the main results of this section.
\begin{thm}\label{palais}
Let $f_k \in W^{2,p}_{{\rm im}}(\Sigma,\R^n)$, $p>2$, be a sequence satisfying 
$$
\mathcal{E}^p(f_k) \le C \quad \mbox{ and } \quad 
||D\mathcal{E}^p(f_k)||_{f_k} \rightarrow 0.
$$
Then, after choosing a subsequence and passing to $f_k \circ \varphi_k$ for suitable diffeomorphisms 
$\varphi_k \in C^\infty(\Sigma,\Sigma)$, the $f_k$ converge strongly in $W^{2,p}(\Sigma,\R^n)$ 
to some $f \in W^{2,p}_{{\rm im}}(\Sigma,\R^n)$, and $f$ is a smooth critical point of $\mathcal{E}^p$.
\end{thm}
\begin{thm}\label{palaisW}
Let $f_k \in W^{2,p}_{{\rm im}}(\Sigma,\R^n)$, $\delta>0$, $p>2$, be a sequence satisfying 
$$
\mathcal{W}^p(f_k) \le C,\quad \mathcal{W}(f_k)\le 8\pi-\delta \quad \mbox{ and } \quad 
||D\mathcal{W}^p(f_k)||_{f_k} \rightarrow 0.
$$
Then, after choosing a subsequence and passing to $f_k \circ \varphi_k$ for suitable diffeomorphisms 
$\varphi_k \in C^\infty(\Sigma,\Sigma)$, the $f_k$ converge strongly in $W^{2,p}(\Sigma,\R^n)$ 
to some $f \in W^{2,p}_{{\rm im}}(\Sigma,\R^n)$, and $f$ is a smooth critical point of $\mathcal{W}^p$.
\end{thm}
Since the arguments for the two results are very similar (thanks to Theorem \ref{p0} and Theorem \ref{comp}) we only present the proof of Theorem \ref{palais}.
\begin{proof}
Langer's compactness theorem \cite{langer85} yields that 
after passing to a subsequence $f_k \circ \varphi_k \to f$ in the $C^1$ topology
and weakly in $W^{2,p}(\Sigma,\R^n)$, where $f \in W^{2,p}_{{\rm im}}(\Sigma,\R^n)$
and $\varphi_k \in C^\infty(\Sigma,\Sigma)$ are diffeomorphisms. It remains 
to see that the convergence is strong in $W^{2,p}(\Sigma,\R^n)$, for which it 
suffices to consider the local convergence of the graph representations 
over a disk $B_r \subset \R^2$. Namely, then the assumption implies that $f$ 
is a critical point of $\mathcal{E}^p$ and is hence smooth by Theorem 
\ref{regularity}, after composing with a further diffeomorphism.\\
\\
Let $u_k,\, u \in W^{2,p}(B_r,\R^{n-2})$ be the graph functions for $f_k$ and $f$, 
respectively. Then $u_k \to u$ in $C^1(B_r)$, weakly in $W^{2,p}(B_r)$, and we can 
assume
\begin{equation}
\label{uk}
||Du_k||_{C^0(B_r)} \le L \leq 1 \quad \mbox{ and }
||u_k||^{p}_{W^{2,p}(B_r)} \le C\, \mathcal{E}^p(f_k) \le C. 
\end{equation}
We let $\psi_k = \eta (u_k-u)$ where $\chi_{B_{r/2}} \leq \eta \leq \chi_{B_r}$ is a 
cut-off function. Clearly
\begin{equation}
\label{psik}
\|\psi_k\|_{C^1(B_r)} \to 0 \quad \mbox{ and } \quad ||\psi_k||_{W^{2,p}(B_r)} \le C.
\end{equation}
Next we recall from \eqref{eqgraphsystem} the Fr\'{e}chet derivative, 
in a graph representation:
\begin{equation}
\label{palais5}
D\mathcal{E}^p(f_k)(0,\psi_k) = 
\int_{B_r} \Big(a^{\alpha \beta}_i (Du_k,D^2 u_k) \partial^2_{\alpha \beta} \psi_k^i 
+ b_i^\alpha(Du_k,D^2 u_k) \partial_\alpha \psi^i_k\Big),
\end{equation}
where 
\begin{eqnarray*}
a_i^{\alpha \beta}(Du_k,D^2 u_k) & = & 
\frac{p}{4}(1+|A_k|^2)^{\frac{p-2}{2}} B^{\alpha \beta, \gamma\lambda}_{ij}(Du_k) 
\partial^2_{\gamma \lambda} u_k^j  \sqrt{\det g_k},\\
b_i^{\alpha}(Du_k,D^2 u_k) & = & 
\frac{p}{8} (1+|A_k|^2)^{\frac{p-2}{2}} 
\frac{\partial B^{\gamma \lambda, \mu \nu}_{jm}}{\partial p_\alpha^i}(Du_k)
\partial^2_{\gamma \lambda}u_k^j \partial_{\mu \nu} u_k^m \sqrt{\det g_k}\\
& & + \frac{1}{4} (1+|A_k|^2)^{\frac{p}{2}} 
\frac{\partial \sqrt{\det g_k}}{\partial p^i_\alpha}(Du_k). 
\end{eqnarray*}
Here $(g_k)_{\alpha \beta} =
\delta_{\alpha \beta} + \langle \partial_\alpha u_k,\partial_\beta u_k\rangle$ and
$B^{\alpha \beta, \gamma \lambda}_{ij}(Du_k) = 
g_k^{\alpha \gamma} g_k^{\beta \lambda} 
(\delta_{ij} - g_k^{\mu \nu} \partial_\mu u_k^i \partial_\nu u_k^j)$.
We see easily that
\begin{eqnarray*}
a_i^{\alpha \beta}(Du_k,D^2 u_k) & \leq & C\,(1+|D^2 u_k|^2)^{\frac{p-1}{2}},\\
b_i^{\alpha}(Du_k,D^2 u_k) & \leq & C\,L\,(1+|D^2 u_k|^2)^{\frac{p}{2}},
\end{eqnarray*}
and obtain for $k \to \infty$
\begin{eqnarray}
\label{eqakinfty}
\int_{B_r} a_i^{\alpha \beta}(Du_k,D^2 u_k)
\Big(\partial^2_{\alpha \beta}\psi_k^i - \eta\,\partial^2_{\alpha \beta} (u^i_k - u^i)\Big)  
& \to \ 0,\\
\label{eqbkinfty}
\int_{B_r} b_i^\alpha(Du_k,D^2 u_k) \partial_\alpha \psi^i_k & \to 0.
\end{eqnarray}
Now using \eqref{uk} and \eqref{psik} we get 
$\|(0,\psi_k)\|_{W^{2,p}_{f_k}(\Sigma)} \leq C\,\|\psi_k\|_{W^{2,p}} \leq C$,
and hence 
$$
D\mathcal{E}^p(f_k)(0,\psi_k) \rightarrow 0  \quad \mbox{ as } k \to \infty,
$$
using the assumption of the theorem and \eqref{psequivariance}. Combining with 
\eqref{eqakinfty} and \eqref{eqbkinfty}, and noting that
$a^{\alpha \beta}_i (Du,D^2 u) \in L^p(B_r,\R^{n-2})'$, 
we conclude
$$
\int_{B_r}\eta \Big(a^{\alpha \beta}_i(Du_k,D^2 u_k) - a^{\alpha \beta}_i(Du,D^2 u)\Big)
\partial^2_{\alpha \beta}(u_k^i-u^i) \rightarrow 0 \quad \mbox{ as } k \to \infty.
$$
But since $u_k \to u$ in $C^1(B_r,\R^{n-2})$ we also have that 
$$
\int_{B_r}\eta \Big(a^{\alpha \beta}_i(Du_k,D^2 u_k) - a^{\alpha \beta}_i(Du,D^2 u_k)\Big)
\partial^2_{\alpha \beta}(u_k^i-u^i) \rightarrow 0, 
$$
and by adding the last two equations we get
$$
\int_{B_r}\eta \Big(a^{\alpha \beta}_i(Du,D^2 u_k) - a^{\alpha \beta}_i(Du,D^2 u)\Big)
\partial^2_{\alpha \beta}(u_k^i-u^i) \rightarrow 0 \quad \mbox{ as } k \to \infty.
$$
Finally we use the ellipticity, see \eqref{eqellipticity}, to estimate
\begin{eqnarray*}
&& \int_{B_r}\eta \Big(a^{\alpha \beta}_i(Du,D^2 u_k) - a^{\alpha \beta}_i(Du,D^2 u)\Big)
\partial^2_{\alpha \beta}(u_k^i-u^i)\\
& = & \int_{B_r}\eta \int_0^1 \frac{\partial a^{\alpha \beta}_i}{\partial q^j_{\lambda \mu}}
\big(Du,D^2 u + t D^2 (u_k - u)\big) \partial^2_{\lambda \mu} (u_k^j - u^j) 
\partial^2_{\alpha \beta}(u_k^i-u^i)\,dt\\
& \geq & \lambda \,\int_{B_r} \eta \int_0^1  \Big(1+|D^2 u + t D^2 (u_k - u)|^2\Big)^{\frac{p-2}{2}}
|D^2 (u_k-u)|^2\,dt\\
& \geq & c\,\lambda \int_{B_{r/2}} |D^2 (u_k-u)|^p.
\end{eqnarray*}
In the last step we used the elementary Lemma $19.27$ from \cite{P}. Altogether we 
have proved local and hence global convergence in $W^{2,p}$. 
\end{proof}

\end{document}